\documentclass[aoas]{imsart}

\RequirePackage{amsthm,amsmath,amsfonts,amssymb}
\RequirePackage[authoryear]{natbib}
\RequirePackage[colorlinks,citecolor=blue,urlcolor=blue]{hyperref}
\RequirePackage{graphicx}
\usepackage{algorithm}
\usepackage{multirow}
\usepackage{algorithmic}
\startlocaldefs

\theoremstyle{plain}

\theoremstyle{definition}

\def\diag{\mbox{diag}}

\endlocaldefs

\begin{document}

\begin{frontmatter}
\title{BASIC: Bipartite Assisted Spectral-clustering for Identifying Communities in Large-scale Networks}

\begin{aug}
\author[A]{\fnms{Tianchen}~\snm{Gao} \ead[label=e1]{gaotc@pku.edu.cn}\orcid{0000-0002-8814-1427}},
\author[B]{\fnms{Jingyuan}~\snm{Liu} \thanks{{Corresponding author}.}\ead[label=e2]{jingyuan@xmu.edu.cn}\orcid{0000-0001-7490-0368}}
\and
\author[C]{\fnms{Rui}~\snm{Pan}\ead[label=e3]{ruipan@cufe.edu.cn}\orcid{0000-0002-6007-4285}}
\author[D]{\fnms{Ao}~\snm{Sun}\ead[label=e4]{ao.sun@marshall.usc.edu}\orcid{0000-0002-8552-7201}}
\address[A]{Beijing International Center for Mathematical Research, Peking University \printead[presep={ ,\ }]{e1}}

\address[B]{MOE Key Laboratory of Econometrics, Department of Statistics and Data Science, School of Economics, Wang Yanan Institute for Studies in Economics,\\ Fujian Key Lab of Statistics, Xiamen University\printead[presep={,\ }]{e2}}

\address[C]{School of Statistics and Mathematics, Central University of Finance and Economics \printead[presep={,\ }]{e3}}

\address[D]{Data Sciences and Operations Department, Marshall School of Business, University of Southern California \printead[presep={,\ }]{e4}}
\end{aug}

\begin{abstract}
Community structure learning is a fundamental task in network analysis, yet it becomes challenging and unstable under weak community signals, which are typical in real-world networks. Motivated by a newly collected large-scale academic network from {\it Web of Science} that intrinsically contains multi-layer information, we propose a new method, named Bipartite-Assisted Spectral-clustering method for Identifying Communities (BASIC), to address this weak-signal issue by leveraging bipartite network information. To the best of our knowledge, BASIC is the first work to integrate auxiliary bipartite information for community detection of the primary network. We thoroughly analyze the new statistical academic collaboration network data using BASIC, and observe that the bipartite network information from author-paper, author-institution, and author-region relationships significantly enhance the identification of interpretable communities of the collaboration network. This enables us to identify meaningful communities, analyze collaboration patterns, trace their evolution, and characterize their research themes. Extensive simulations validate the empirical performance of BASIC. Theoretically, under a degree-corrected stochastic block model framework, we prove that BASIC possesses superior accuracy and stability, achieving a tighter upper error bound than methods using only the primary network, especially in weak-signal regimes.

\end{abstract}


\end{frontmatter}

\section{Introduction}

Understanding the organization of scientific collaboration has become an important topic across statistics, scientometrics, data science, and others \citep{newman2001scientific,ji2022co,gao2024citation}. The rapid growth of bibliographic data from digital databases such as the {\it Web of Science} and {\it Google Scholar} enables quantitative investigations into how researchers, institutions, and regions interact in producing scientific knowledge \citep{newman2001scientific,newman2008physics,ji2016coauthorship,ji2022co}. Representing such relationships as networks provides a powerful statistical framework for studying knowledge diffusion, identifying collaborative communities, and mapping disciplinary structures \citep{girvan2002community, newman2012communities,fortunato2010community}. However, large-scale academic collaboration networks are typically sparse, noisy, and highly heterogeneous. Many authors have few collaborators and a small number of prolific researchers dominate the degree distribution \citep{newman2004coauthorship}. This sparsity and heterogeneity, combined with weak signals between communities, make the detetion of meaningful community structures a challenging statistical problem \citep{karrer2011stochastic,jin2015fast,Jin_2021}.

During network analysis, community detection is a crucial tool that focuses on identifying closely connected groups within networks \citep{girvan2002community,newman2012communities,jin2015fast}. An extensively used modeling strategy for network analysis is the stochastic block model (SBM) where the edge probabilities depend only on the communities to which the correponding nodes belong \citep{holland1983stochastic}. Given the diversity and heterogeneity of real-world networks, various extensions of the SBM have been proposed, including the degree-corrected block model (DCBM) \citep{karrer2011stochastic}, mixed membership stochastic block model (MMSB) \citep{airoldi2008mixed, zhang2020detecting}, degree-corrected mixed membership model (DCMM) \citep{jin2017estimating}, bipartite stochastic block model (BiSBM) \citep{larremore2014efficiently, yen2020community}, superimposed stochastic block model (SupSBM) \citep{huang2020efficient, paul2023higher}, motif tensor block model (MoTBM) \citep{yu2024network}, and so forth.

A variety of methods have been proposed for estimating community structures under the SBM framework. Among them, spectral clustering is particularly appealing due to its scalability and clear theoretical foundation \citep{Rohe_2011,jin2015fast}. Specifically, the SCORE method proposed by \cite{jin2015fast} considers ratios of eigenvectors to correct degree heterogeneity and achieves consistent estimation under DCBM. \cite{ji2016coauthorship} and \cite{ji2022co} further extended SCORE to directed-DCBM and DCMM. Several methods incorporate auxiliary covariates into network community detection. {Most of these studies consider a single network accompanied by nodal covariates, combining the two sources through a regularized graph Laplacian \citep{binkiewicz2017covariate}, a semidefinite program involving a covariate kernel \citep{yan2021covariate}, a likelihood-based block model for community membership \citep{weng2022community}, or degree-adaptive network adjustment for multiscale networks \citep{hu2024network}. In the multilayer setting, \cite{xu2023covariate} constructed an additional layer from node-wise covariates and incorporated it through Tucker decomposition. Pairwise covariates-adjusted block model considers a different setting involving pairwise rather than nodal covariates and explicitly models and adjusts their effects on edge intensities \citep{huang2024pcabm}.}

Despite these advances, community detection remains challenging when the network has a low signal-to-noise ratio, as the connection probabilities within and between communities become similar. This weak-signal regime is common in empirical social and collaboration networks, where random noise can easily mask the underlying block structure \citep{Jin_2021}. Although recent methods such as SCORE+ \citep{Jin_2021} and stepwise model selection procedures \citep{jin2023optimal} enhance inference through spectral normalization and eigenvalue diagnostics, they still rely solely on the primary adjacency matrix and often ignore auxiliary information that could aid statistical estimation. In practice, however, networks rarely exist in isolation. Beyond the primary network of interest, there are often bipartite or multi-layer networks that carry related relationships.

In this paper, we focus on academic collaboration in the field of statistics, i.e., a network where authors are connected through joint publications \citep{zhang2023community,ji2016coauthorship}. We collect a large-scale academic collaboration network dataset from {\it Web of Science}, covering publications from 42 statistical journals between 1981 and 2021. Our goal is to identify the community structure of this network, analyze collaboration patterns within and across research groups, and trace their evolution over time.
From a science-of-science perspective \citep{fortunato2018science}, such an analysis is valuable because it helps identify closely related collaborators, understand the structure and dynamics of research communities, and place individual work in a broader scientific context. This is particularly important in the age of AI, where the rapid growth of scientific publications makes it increasingly difficult to track relevant communities, emerging topics, and related research directions  \citep{schafermeier2023research}.

In addition to the collaboration network, we also obtain an author–paper network connecting researchers to their publications, an author–institution network reflecting affiliations, and an author–region network capturing geographical patterns. These bipartite networks contain valuable information that can help detect communities from complementary perspectives. Intuitively, authors affiliated with the same institution or geographic area are more likely to collaborate, and thus share the same community memberships. Incorporating such bipartite relations can therefore improve statistical efficiency and reduce the weaknesses of the primary network signal. Figure~\ref{fig:Target_Source} presents part of these networks mentioned above. However, integrating such diverse networks in a statistically principled way is nontrivial, as they may differ in size, structure, and sparsity. Na\"ive aggregation may disrupt the underlying community structure or cause negative transfer. As a result, developing a theoretically justified method that unifies multiple network sources while preserving interpretability remains an open methodological challenge.

\begin{figure}[!ht]
    \centering
    \includegraphics[width=0.9\textwidth]{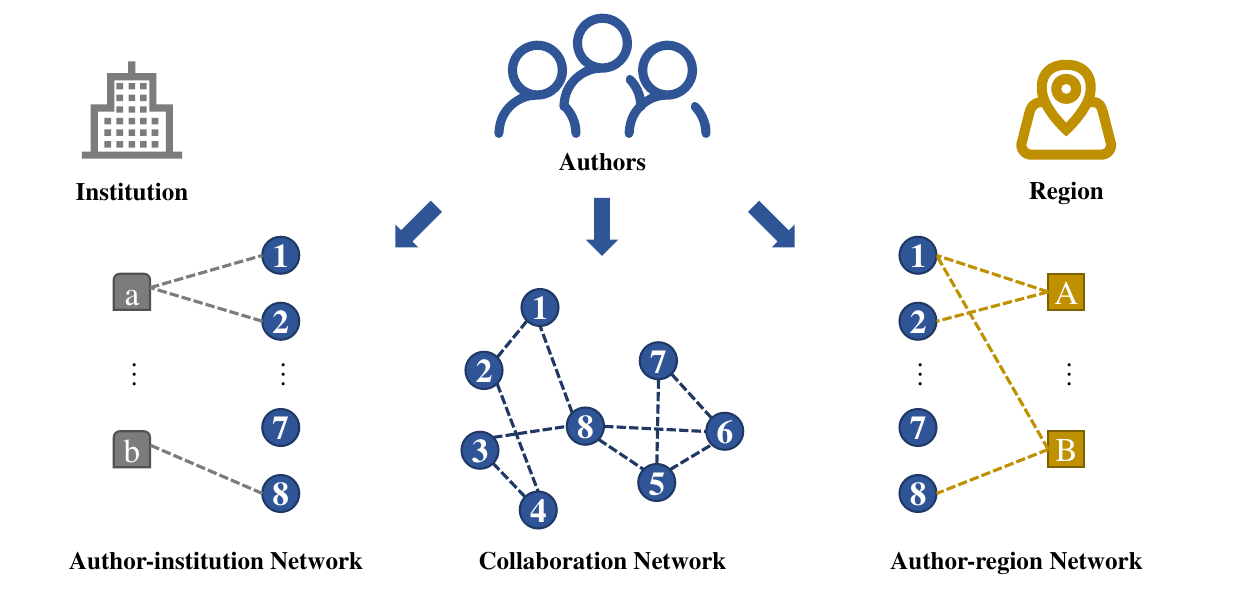}
    \caption{Schematic diagram of the collaboration network, author-institution network, and author-region network. The structural information of author nodes within these bipartite networks can be leveraged to aid in detecting communities within the collaboration network.}
    \label{fig:Target_Source}
\end{figure}

Motivated by the four interrelated networks, we propose a unified framework called {\it Bipartite Assisted Spectral-clustering for Identifying Communities (BASIC)} to enhance community detection in weak-signal networks by incorporating information from auxiliary bipartite networks. Building upon the DCBM and and its bipartite modification BiDCBM, we aggregate adjacency matrices from both the primary and bipartite networks into a unified spectral representation, followed by eigenvalue decomposition and SCORE normalization \citep{ke2022score}. This design allows BASIC to automatically correct for degree heterogeneity while effectively utilizing bipartite information to strengthen the community signal. We rigorously establish the theoretical properties of the proposed method, showing that the aggregated spectral formulation preserves the original community structure, avoids negative transfer, and yields a strictly tighter upper bound on the mis-clustering error compared with using the primary network alone. Furthermore, even under extremely weak signals, the presence of one informative bipartite network guarantees faster convergence of the estimated community labels. The theoretical framework does not require independence between the primary network and the bipartite networks. This is a key distinction between BASIC and existing methods based on multi-layer networks \citep{lei2020consistent} and tensor networks \citep{han2022optimal}. These results provide new insights into how auxiliary relational data can enhance the understanding and recovery of latent community structures. {Unlike the aforementioned covariate-based approaches, BASIC directly incorporates multiple bipartite networks with potentially different auxiliary node sets, rather than representing auxiliary information as nodal features or pairwise regressors. It aggregates structural information from the primary and bipartite networks into a unified spectral representation and then applies SCORE normalization to account for degree heterogeneity. BASIC complements the covariate-based literature by focusing on how auxiliary relational networks can strengthen community recovery when the primary network has a weak signal.} 

We utilize the proposed BASIC to the aforementioned large-scale collaboration dataset in the field of statistics. By jointly analyzing the primary author–author collaboration network and three bipartite networks, including the author-paper network, the author-institution network, and the author-region network, BASIC reveals coherent and interpretable research communities, capturing collaboration patterns among scholars and institutions across different regions. For example, BASIC identifies several representative research communities in statistics, such as those focusing on biostatistics, high-dimensional data analysis, and statistical applications. These communities exhibit clear internal collaboration patterns and strong institutional or regional cohesion. The detected structures illustrate how BASIC effectively captures both thematic and geographic clustering in academic collaborations. This application highlights the capability of BASIC to extract meaningful scientific structures.

In summary, this paper contributes to the network structure learning from all three aspects - methodological, practical and theoretical. Methodologically, to the best of our knowledge, this is the first study to utilize bipartite network information to improve community detection performance. In contrast, existing research focuses primarily on extracting additional insights from the primary network itself, such as nodal covariates or higher-order structures.
It is important to note that the proposed BASIC can readily be integrated with existing methods to further tackle weak network signals.
Practically, we apply BASIC to a newly collected large-scale academic collaboration dataset that naturally includes three bipartite networks. This analysis provides meaningful insights into collaboration patterns from the scientific perspective. Meanwhile, BASIC is applicable to multi-relational networks where bipartite information can improve the recovery of stable and interpretable communities. Theoretically, we rigorously establish the theoretical guarantee of BASIC by deriving its convergence rate even under extremely weak signals of the primary network. More importantly, we prove that BASIC has a tighter upper bound of the number of error nodes compared to merely using the primary network for community detection, and hence enhances detection power.

The remainder of the paper is organized as follows. Section~\ref{Data} introduces the dataset along with some descriptive analysis. In Section~\ref{Methodology}, we introduce the methodology of BASIC. Section~\ref{realdata} reports the real data analysis on structure learning of the author collaboration network in the field of statistics. Section~\ref{Simulation} presents simulation studies to demonstrate the finite sample performance of BASIC. We conclude with a discussion in Section~\ref{conclusion}. The theoretical guarantees of BASIC are provided in the Supplementary Material \citep{gao2026basic_supp}.

\section{Data and Descriptive Analysis}\label{Data}

\subsection{Data Collection and Network Construction}

We collect publications from 42 renowned statistical journals from 1981 to 2021 in {\it Web of Science} ({\it www.webofscience.com}). We obtain titles, abstracts, keywords, citation counts up to 2022, years, journals, author information (including names, institutions, and regions) and reference lists, as illustrated in Table~\ref{tab:data_example}. After an extensive data cleaning and author disambiguation process, we address several sources of ambiguity, including inconsistent name formats (e.g., ``Wang, H." vs. ``Wang, Hua"), identical abbreviations referring to different individuals (e.g., ``Ding, Y." for multiple authors), institutional changes over time, and duplicate names across institutions. To resolve these issues, we implemente a systematic four-step identification procedure combining string similarity matching, institutional affiliation comparison, graph-based connected component merging, and manual verification for common surnames \citep{zhang2023community}.

\begin{table}[!ht]
\caption{An illustrative example of a statistical publication. Basic information such as title, abstract, keywords, citation counts, author information, and reference list are obtained. The citation counts are provided by the {\it Web of Science} up to 2022.}
\newcommand{\tabincell}[2]{\begin{tabular}{@{}#1@{}}#2\end{tabular}}
\label{tab:data_example}
\begin{center}
\begin{tabular}{ll}
    \hline
    \textbf{Variable} & \textbf{Example} \\
    \hline
    Title & \tabincell{l}{A proportional hazards model for the subdistribution of \\a competing risk}\\
    Abstract &  \tabincell{l}{With explanatory covariates, the standard analysis for \\competing risks data involves modeling the cause-specific\\ hazard functions via a proportional hazards assumption...}\\
    Keywords & \tabincell{l}{hazard of subdistribution; martingale; partial likelihood;\\ transformation model}\\
    Journal & Journal of the American Statistical Association\\
    Year & 1999\\
    Citation counts & 8,265 (until 2022)\\
    Author information & \tabincell{l}{Fine, Jason P.@
University of Pittsburgh;\\ Gray, Robert J.@Harvard University} \\
    Reference list &  \tabincell{l}{$\left[1\right]$ Nonparametric estimation of partial transition-probabilities \\in multiple decrement models@Aalen, O@1978@AoS\\ $\left[2\right]$ Estimates of absolute cause-specific risk in cohort studies\\Benichou, J \& Gail, MH@1990@Biometrics \\$\left[3\right]$ ...}\\
    \hline
    \end{tabular}
    \end{center}
\end{table}

After cleaning, we construct a collaboration network with 16,125 nodes and 22,530 edges as the original network, i.e., the nodes represent the authors, with an undirected and unweighted edge between two authors if they have published two or more papers together \citep{zhang2023community,ji2022co}. The resulting collaboration network has a density of 0.017\%, indicating that it is very sparse. To concentrate on the most important nodes, we extract the $c$-core network by iteratively removing nodes with a degree less than $c$ until the network stabilizes \citep{2016Discussion}. The $c$-core network is a commonly used method for extracting core information from networks \citep{miao2023informative,ding2023matrix}.
Specifically, we extract the 4-core of the collaboration network and obtain the largest connected component, resulting in a core network with 737 nodes and 2,453 edges, with a network density of 0.904\%. This core network serves as the primary network for community detection.

In addition, author information can help us obtain the corresponding institution and region of the authors. Therefore, we consider three bipartite networks, i.e., the author-paper network, the author-institution network, and the author-region network. In total, they involve 17,722 papers, 295 institutions, and 44 regions. In Table \ref{tab:bipartite_information}, we present the the top 5 institutions and regions with the highest frequencies.
In the author–paper network, an edge connects an author to a paper if the author has contributed to that publication.
The author–institution network links each author to their affiliated institutions, reflecting academic and organizational associations that often facilitate collaboration.
Finally, the author–region network connects authors to the geographical regions of their affiliations, allowing us to characterize regional clustering and cross-regional collaborations.
Together, these bipartite networks provide complementary relational information that enriches the understanding of collaboration structures beyond the primary collaboration network.

\begin{table}[!ht]
\caption{The information of top 5 institutions and regions with the highest frequencies.}
\label{tab:bipartite_information}
\centering
\begin{tabular}{ccc|ccc}
\hline
\textbf{ID} & \textbf{Institution} & \textbf{Frequency} & \textbf{ID} & \textbf{Region} & \textbf{Frequency} \\
\hline
1 & University of North Carolina & 22 & 1 & United States & 305 \\
2 & Universidade de Sao Paulo & 20 & 2 & China & 104 \\
3 & Harvard University & 19 & 3 & Brazil & 44 \\
4 & Hasselt University & 18 & 4 & Belgium & 39 \\
5 & Texas A\&M University College Station & 14 & 5 & Spain & 34 \\
\hline
\end{tabular}
\end{table}

We note that institutional and regional affiliations of authors may change over time, especially given the long study period from 1981 to 2021. In this work, we assign each author a main institution and region, defined as the institution and region that appear most frequently for that author in our dataset. When constructing the author–institution and author–region networks, this approach provides a simple and stable representation of institutional and regional identity. We then define the main-institution ratio as the proportion of institution records for an author that correspond to the main institution.
In our dataset, 81\% of authors have a main-institution ratio above 95\%, while only 3\% of authors have a main-institution ratio below 50\%. Most authors are therefore highly concentrated in a single institution, suggesting that the main-institution and region assumption is reasonable for the majority of cases.

\subsection{Descriptive Analysis}

To better understand the structural characteristics of the academic collaboration network, a descriptive analysis is conducted. Figure \ref{fig:degreedis} illustrates the degree distribution of the core collaboration network. The left panel shows the histogram of node degrees, while the right panel presents the log–log relationship between node degree and the number of authors. Both plots reveal a highly skewed distribution, where most authors have very few collaborators, whereas a small fraction of authors exhibit extremely high degrees. This long-tailed pattern indicates strong heterogeneity in the connectivity of the network and approximately follows a power-law distribution, a key feature of many real-world collaboration networks \citep{ding2023matrix}. In such networks, a few highly connected researchers act as hubs that dominate the collaboration structure, while the majority of authors maintain only limited local connections. The pronounced degree heterogeneity poses challenges for community detection, as traditional algorithms may be biased toward these hubs \citep{jin2015fast,ke2022score}, masking the true underlying community structure. Accounting for this variability in node degree is therefore crucial for accurate and interpretable clustering of academic collaboration networks.

\begin{figure}[!ht]
    \centering
    \includegraphics[width=0.98\linewidth]{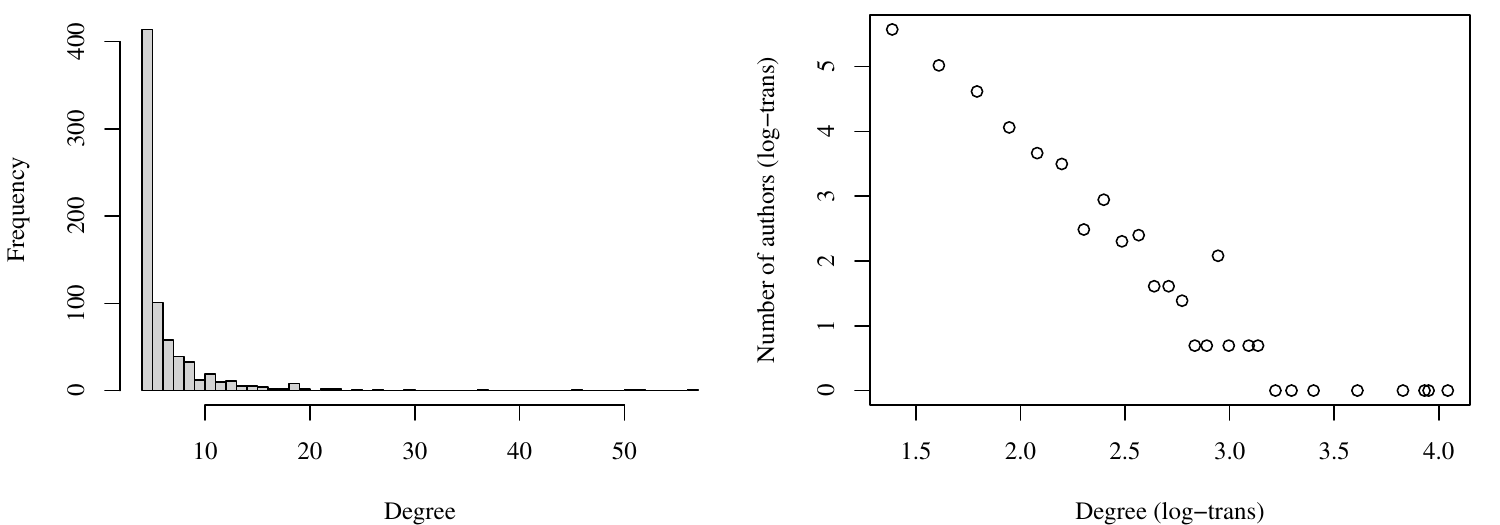}
    \caption{Histogram and log–log plot of degree distribution in the core collaboration network}
    \label{fig:degreedis}
\end{figure}

While the degree distribution reveals strong heterogeneity in collaboration intensity, the overall structural organization of these collaborations can be better understood through network visualization. Figure \ref{fig:collaboration_network} presents the visualization of the core network of the academic collaboration network, where nodes represent authors and edges indicate collaborative relationships. Node sizes correspond to the number of collaborators, and the top 30 most connected authors are shown in dark blue. The network reveals a moderate level of community organization, consisting of several locally dense clusters that are loosely connected through a few hub nodes. Nevertheless, a noticeable intermixing of nodes across clusters renders the community boundaries less pronounced. Overall, the visualization offers an informative depiction of the organizational pattern of academic collaboration network in statistics, where the community structure appears complex and not sharply delineated.

\begin{figure}[!ht]
    \centering
    \includegraphics[width=0.6\linewidth]{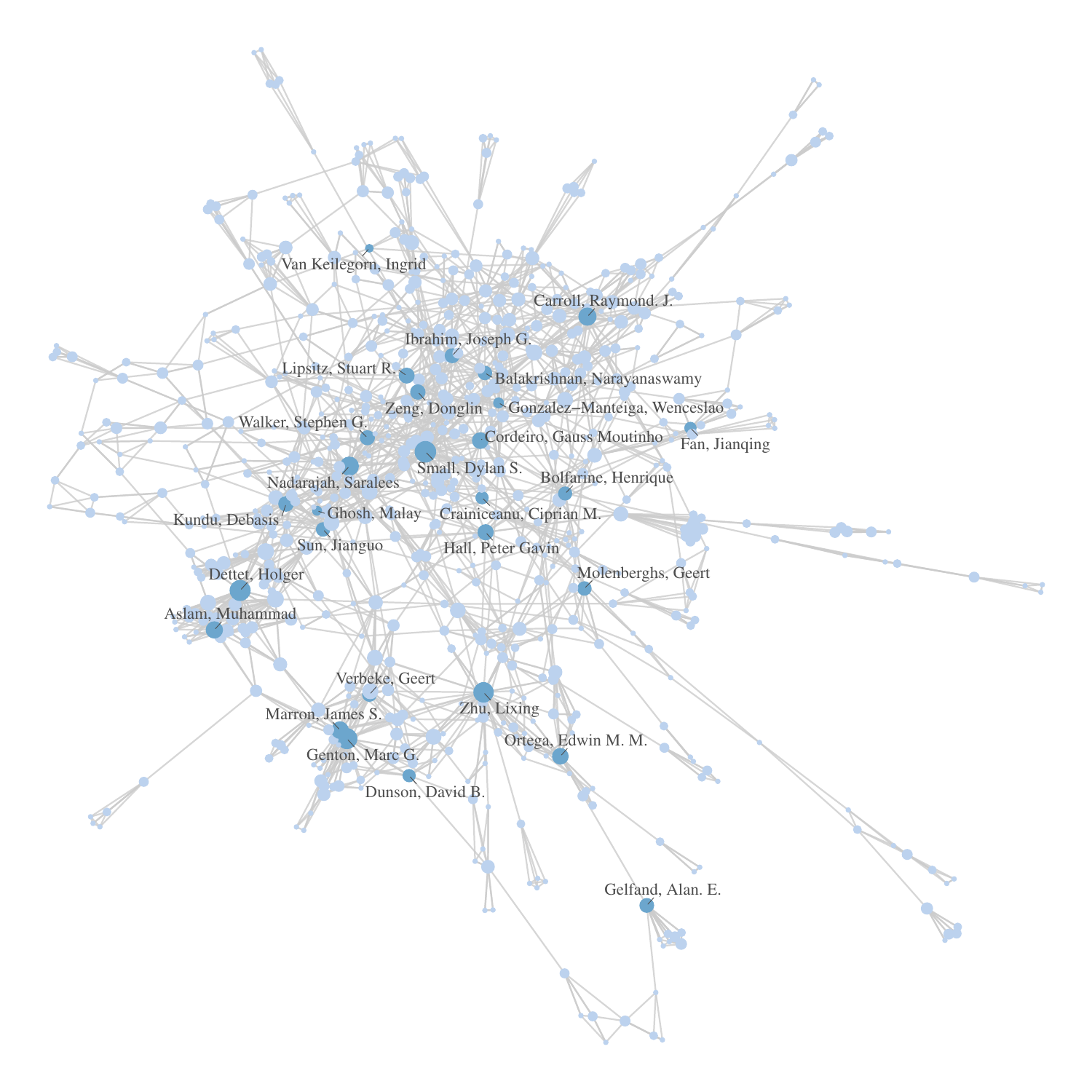}
    \caption{Visualization of the core network of the academic collaboration network, illustrating its modular and community structure}
    \label{fig:collaboration_network}
\end{figure}

Preliminary examination of the collaboration network reveals several distinctive statistical features. First, the degree distribution is highly right-skewed, indicating that a small number of authors account for a large share of collaborations, whereas most authors have only a few collaborators. Second, the overall network density is extremely low (0.017\%), and even after extracting the 4-core subgraph, the resulting core network remains sparse, suggesting that collaboration patterns in the statistics community are locally clustered but globally fragmented. Third, the eigen-gap between the 12-th and 13-th largest eigenvalues of the adjacency matrix is small, as shown in Figure~\ref{fig:scree_plot_collaboration}, implying that the underlying community structure is statistically weak and difficult to recover through the primary network alone. These observations highlight a fundamental challenge: although collaboration networks can reveal important patterns of scientific collaboration, weak network signals and strong degree heterogeneity hinder reliable community detection. Moreover, the availability of rich auxiliary data, including author–paper linkages, institutional affiliations, and geographic information, suggests that these bipartite relations may contain complementary structural information about research communities. Incorporating such auxiliary information into network modeling can therefore enhance detection stability and interpretability.

Motivated by these observations, we next introduce a statistical framework that integrates bipartite and primary networks for community structure learning. The proposed approach, BASIC, provides a principled spectral method that effectively aggregates bipartite network information and improves estimation accuracy even under weak-signal conditions. The methodology is presented in the following section.

\section{Methodology of BASIC}\label{Methodology}
\subsection{Notation}
For a vector $\mathbf{v}$ and a fixed integer $q > 0$, $\mathbf{v}(i)$ or $v_i$ stands for its $i$-th component, $\|\mathbf{v}\|_q$ represents its $\ell_q$ norm; when $q = 2$, we simplify it as $\|\mathbf{v}\|$. When referring to a matrix $\mathbf{M}$, $\mathbf{M}(i, j)$ denotes its $(i,j)$-th entry, $\mathbf{M}_{\bar{i}}$ represents its $i$-th row, $\mathbf{M}_{i \sim j}$ refers to the sub-matrix by removing the $i$-th to $j$-th columns of $\mathbf{M}$. Furthermore, $\|\mathbf{M}\|_F$ represents its Frobenius norm, defined as the square root of the sum of the squares of all its elements. Define $\sigma_{\max}(\mathbf{M})$, $\sigma_{\min}(\mathbf{M})$ and $\sigma_i(\mathbf{M})$ as the largest, smallest, and $i$-th leading absolute singular value of $\mathbf{M}$, respectively. If $\mathbf{M}$ is a square matrix, denote $\lambda_{\max}(\mathbf{M})$, $\lambda_{\min}(\mathbf{M})$ and $\lambda_i(\mathbf{M})$ to be its largest, smallest, and $i$-th leading eigenvalue in absolute values. If further $\mathbf{M}$ is symmetric, $\|\mathbf{M}\|_{op}$ represents its operator norm which equals the largest absolute value of its eigenvalues.
For a set $\mathcal{V}$, $|\mathcal{V}|$ denotes its cardinality. When we have two positive sequences $\left\{a_n\right\}_{n=1}^{\infty}$ and $\left\{b_n\right\}_{n=1}^{\infty}$, use $a_n \asymp b_n$ to indicate the existence of a constant $C$ such that for large enough $n$, $b_n / C \leq a_n \leq C b_n$, i.e.,  $a_n$ and $b_n$ are of the same order. Denote $ a_n \gg b_n $ if ${a_n}/{b_n} \to \infty$. Finally, define $\mathbb{I}(A)$ to be the indicator function of event $A$.

\subsection{DCBM and BiDCBM for primary and bipartite networks}\label{model-section}

Denote $\mathcal{G}_{t}=(\mathcal{V}_{t}, \mathcal{E}_{t})$ to be the primary network of interest, where $\mathcal{V}_{t}$ and $\mathcal{E}_{t}$ represent the set of primary nodes and edges, respectively. That is, $\mathcal{V}_{t}=\left\{1, \ldots, n\right\}$, with $n$ being the number of primary nodes in $\mathcal{G}_{t}$, and $\mathcal{E}_{t}$ collects the node pairs $(i, j)$ if node $i$ and node $j$ are connected. Define the adjacency matrix of $\mathcal{G}_{t}$ as $\mathbf{A}=\left(\mathbf{A}{(i_1,i_2)}\right) \in \mathbb{R}^{n \times n}$, where $i_1,i_2 =1, \ldots,n$, that is, $\mathbf{A}{(i_1,i_2)}=\mathbb{I}((i_1,i_2) \in \mathcal{E}_{t})$.
Without loss of generality, set $\mathbf{A}{(i,i)}=0$ for $i = 1, \ldots,n$, indicating the absence of self-connections, although the proposed BASIC method can be easily extended to self-connected networks. We utilize the DCBM as the generating model for $\mathcal{G}_{t}$, so that the probability of an edge between two nodes depends only on the communities to which they belong and the degree heterogeneity parameters. Assume that the primary nodes are partitioned into $K$ distinct communities, where $K$ is fixed. For $k=1, \ldots, K$, denote $\mathcal{V}_{t}^{(k)}$, with cardinality $n_k=|\mathcal{V}_{t}^{(k)}|$, as the set of node indices in the $k$-th community such that $\mathcal{V}_{t} = \bigcup_{k=1}^{K} \mathcal{V}_{t}^{(k)}$ and $\mathcal{V}_{t}^{(k_1)} \cap \mathcal{V}_{t}^{(k_2)} = \emptyset$ for $k_1 \neq k_2$, hence $\sum_{k=1}^K n_k=n$. Let $\boldsymbol{l}= \left(l_1, \ldots, l_n\right)^{\top}$ be the vector of node membership labels, whose elements take values in $\{1, \ldots, K\}$. The DCBM assumes that the probability of an edge between two nodes is $\mathbf{P}\left(\mathbf{A}{(i_1,i_2)}=1\right)={\theta}_{i_1} {\theta}_{i_2} \mathbf{E}{(l_{i_1},l_{i_2})}$, where $\mathbf{E} \in \mathbb{R}^{K \times K}$ is a symmetric probability transition matrix between communities with elements in $[0,1]$, and $0\leq {\theta}_{i}\leq 1$ is the degree heterogeneity parameter associated with primary node $i$. A node with a larger ${\theta}_{i}$ value is more likely to connect with other nodes.

In addition, consider $Q$ bipartite networks that are related to the primary network $\mathcal{G}_{t}$, denoted as $\mathcal{G}^{(q)}_{t,s}=(\mathcal{V}_{t}, \mathcal{V}^{(q)}_{s}, \mathcal{E}^{(q)}_{t,s})$ with $q=1,\ldots,Q$. Specifically, the target nodes set $\mathcal{V}_{t}$ still refers to the set of nodes from the original primary network, while the bipartite nodes set $\mathcal{V}^{(q)}_{s}$ represents the $q$-th set of bipartite nodes. We write $\mathcal{V}_{s}^{(q)}=\left\{1, \ldots, m^{(q)}\right\}$, where $m^{(q)}$ is the number of bipartite nodes in $q$-th bipartite network. Additionally, $\mathcal{E}^{(q)}_{t,s}$ collects the edges in $q$-th bipartite network, which varies among different bipartite networks.
For a specific bipartite network $\mathcal{G}_{t,s}^{(q)}$, we can construct a bipartite adjacency matrix $\mathbf{B}^{(q)}=\left(\mathbf{B}^{(q)}{(i,j)}\right) \in \mathbb{R}^{n \times {m^{(q)}}}$, where $\mathbf{B}^{(q)}{(i,j)}=\mathbb{I}((i,j) \in \mathcal{E}^{(q)}_{t,s})$, for $i = 1, \ldots, n$ and $j = 1, \ldots, m^{(q)}$. We assume that the bipartite nodes can be partitioned into $K^{(q)}$ distinct communities. Let $\boldsymbol{r}^{(q)}=\left(r_1^{(q)}, \ldots, r_m^{(q)}\right)^{\top}$ be the corresponding membership label vector, where $r_j^{(q)}$ takes values in $\{1, \ldots, K^{(q)}\}$.
To characterize the generative mechanism of bipartite networks, we modify the BiSBM \citep{yen2020community} by introducing the degree heterogeneity parameters, and the resulting model is named bipartite degree-corrected block model (BiDCBM). Thus, the probability of an edge between two nodes is $\mathbf{P}\left(\mathbf{B}^{(q)}{(i,j)}=1\right)={\theta}_i {\delta}^{(q)}_j \mathbf{F}^{(q)}{(l_i,r^{(q)}_j)}$, where $\mathbf{F}^{(q)} \in \mathbb{R}^{K \times K^{(q)}}$ is an asymmetric probability transition matrix whose elements are in $[0,1]$, and $0\leq {\theta}_{i}, {\delta}^{(q)}_{j} \leq 1$ represent the degree heterogeneity parameters associated with primary node $i$ and bipartite node $j$ in $q$-th bipartite network, respectively.

\subsection{BASIC: Bipartite Assisted Spectral-clustering for Identifying Communities}

Based on the primary network $\mathcal{G}_{t}$ and $Q$ bipartite networks defined in section \ref{model-section}, the target is to identify the community structure of primary network, hence, it is important to find an efficient way to aggregate information from the primary and all bipartite networks, even if its signal strength is weak. Solving this problem is not straightforward. For instance, how to tackle the different dimensions among the primary and all bipartite networks? How to extract the community structure information of the bipartite nodes? How to effectively integrate the bipartite information while avoiding the ``negative knowledge transfer"? Facing these challenges, we propose a Bipartite Assisted Spectral-clustering method for Identifying Communities, abbreviated as BASIC, by constructing an {\it aggregated square matrix}
\begin{equation}\label{aggregation_frame}
\mathbf{M} = \mathbf{A} \mathbf{A}^{\top} + \sum_{q=1}^Q \mathbf{B}^{(q)} \mathbf{B}^{(q)\top}.
\end{equation}

The subsequent procedures are applied to this aggregated square matrix $\mathbf{M} \in \mathbb{R}^{n \times n}$ rather than the original scale of adjacency matrices. The formulation of \eqref{aggregation_frame} is automatically adapted to different dimensions of individual adjacency matrices from primary and bipartite networks. Moreover, as elaborated in {Proposition~1} in the Supplementary Material \citep{gao2026basic_supp}, we prove that this aggregation formulation does not disrupt the community structure of the primary network. This is of utmost importance when integrating side information into the primary objective, which ensures no negative transfer, meaning that the aggregated method performs at least as well as solely using the primary information. Last but not least, this formulation indeed guarantees information enhancement and effectively increases detection power, as shown in Theorem~A.2 in the Supplementary Material \citep{gao2026basic_supp}.
Thus, even if the signal from the primary network, represented by its smallest singular value, is extremely weak, the only requirement for successful community identification is the presence of at least one auxiliary bipartite network with a moderately strong signal. Under these conditions, the convergence rate of community identification is strictly faster than that relying solely on the primary network. All these provide a solid foundation for leveraging diverse data sources to enhance the accuracy and effectiveness of the primary network analyses.

Next, we apply eigenvalue decomposition to $\mathbf{M}$, and extract the first $K$ leading eigenvectors, denoted as $\hat{\mathbf{U}} = \left[\hat{\mathbf{u}}_1, \hat{\mathbf{u}}_2,\ldots,\hat{\mathbf{u}}_K\right] \in \mathbb{R}^{n \times K}$, which correspond to the community structure of the primary nodes. Subsequently, we apply the SCORE normalization \citep{jin2015fast} to obtain the ratio matrix
$\hat{\mathbf{R}}  \in \mathbb{R}^{n \times (K-1)}$, which adjusts to degree heterogeneity by taking the ratio of eigenvectors to the eigenvector corresponding to the largest eigenvalue. Specifically, the elements of the ratio matrix $\hat{\mathbf{R}}$ are defined as
\begin{equation}\label{ratio-matrix}
\hat{\mathbf{R}}(i, k) =
        \operatorname{sgn}\left(\hat{\mathbf{u}}_{k+1}(i) \right) \cdot \min \left\{\left|{\hat{\mathbf{u}}_{k+1}(i)}/{\hat{\mathbf{u}}_{1}(i)}\right|, T_n\right\},
\end{equation}
where $\operatorname{sgn}(x)$ stands for the sign function, i.e., $\operatorname{sgn}(x)=1$ when $x>0$, $\operatorname{sgn}(0)=0$, and $\operatorname{sgn}(x)=-1$ when $x<0$. In addition, $T_n$ is a threshold, generally set to $\log(n)$. Last, we conduct the $k$-means algorithm on the rows of the ratio matrix $\hat{\mathbf{R}}$ and obtain the detected communities. Algorithm~\ref{Alg:BASIC} summarizes the step-by-step procedure of BASIC.

\begin{algorithm}[!ht]
    \caption{The procedure of BASIC for community detection}
    \label{Alg:BASIC}
    \begin{enumerate}
        \item[\textbf{Input}:] The adjacency matrix of the primary network $\mathbf{A}$, the number of communities $K$ for the primary nodes, the adjacency matrices of bipartite networks $\mathbf{B}^{(q)}$ for $q = 1, \ldots, Q$.
        \item[\textbf{Step 1}:] Compute the aggregated matrix $\mathbf{M}$ by \eqref{aggregation_frame}.

        \item[\textbf{Step 2}:] Apply the eigenvalue decomposition to  $\mathbf{M}$ and obtain the first $K$ leading eigenvectors $\hat{\mathbf{U}} = \left[\hat{\mathbf{u}}_1, \hat{\mathbf{u}}_2,\ldots,\hat{\mathbf{u}}_K\right] \in \mathbb{R}^{n \times K}$.

        \item[\textbf{Step 3}:] Compute the ratio matrix $\hat{\mathbf{R}} \in \mathbb{R}^{n \times (K-1)}$ in \eqref{ratio-matrix}.
        \item[\textbf{Step 4}:] Apply the $k$-means algorithm to the rows of $\hat{\mathbf{R}}$, and solve for
        $$
        \mathbf{N}^*=\underset{\mathbf{N} \in \mathcal{N}_{n, K-1, K}}{\operatorname{argmin}}\|\mathbf{N}-\hat{\mathbf{R}}\|_F^2,
        $$
        where $\mathcal{N}_{n, K-1, K}$ represents the set of $n \times (K-1)$ matrices with only $K$ distinct rows.
       \item[\textbf{Step 5}:] Utilize $\mathbf{N}^*$ to assign the membership of primary nodes, $\hat{\boldsymbol{l}} = (\hat{l}_{1}, \ldots, \hat{l}_{n})^{\top}$.
        \item[\textbf{Output}:] Community label vector $\hat{\boldsymbol{l}} = (\hat{l}_{1}, \ldots, \hat{l}_{n})^{\top}$.
    \end{enumerate}
\end{algorithm}

\section{Real Analysis of Large-scale Collaboration Network Data}\label{realdata}

This real-data study addresses an important question in statistics and scientometrics \citep{ji2022co,ke2022score}: in a large-scale academic collaboration network, where coauthorship links are often sparse and weak, can naturally available bipartite relations (e.g., author--paper, author--institution, and author--region) improve the identification of stable and interpretable collaboration communities? This question is practically relevant because collaboration patterns in real academic networks are shaped not only by  coauthorship links, but also by publication patterns, institutional and regional structure. The newly collected {\it Web of Science} dataset provides a natural setting for this analysis, as it contains these multiple types of complementary relational information. 

\subsection{Community Detection by BASIC}

In this subsection, we use information from three bipartite networks, from the perspectives of papers, institutions, and regions, to assist the community structure learning of the primary collaboration network. The SCORE method applied to the collaboration network is treated as the baseline. The number of communities $K$ in the primary network is determined by the edge cross-validation (ECV) algorithm \citep{li2020network}. Specifically, we set the maximum number of communities to 30 according to the scree plot as illustrated in Figure~\ref{fig:scree_plot_collaboration}. By repeating the ECV process 20 times for $K$ from 1 to 30, we obtain the optimal number of communities as 12\footnote{We set the parameters of ECV by convention, with the number of samplings set to 3 and the proportion of holdout nodes set to 0.1.}. According to \cite{Jin_2021}, the eigenvalues of a weak-signal network often have the $K$-th and $(K+1)$-th values that are ``close". When $K=12$, we compute the quantity $1 - {\hat{\lambda}_{K+1}}/{\hat{\lambda}_K} = 0.011$, which is smaller than the commonly used scale-free threshold of 0.1. This result highlights that the collaboration network is indeed a weak-signal network. In contrast, for typical strong signal networks, this quantity exceeds 0.1, such as in the Karate (0.414) \citep{zachary1977information}  and Polblogs (0.6) \citep{adamic2005political} networks.

\begin{figure}[!htp]
    \centering
    \includegraphics[width=0.6\textwidth]{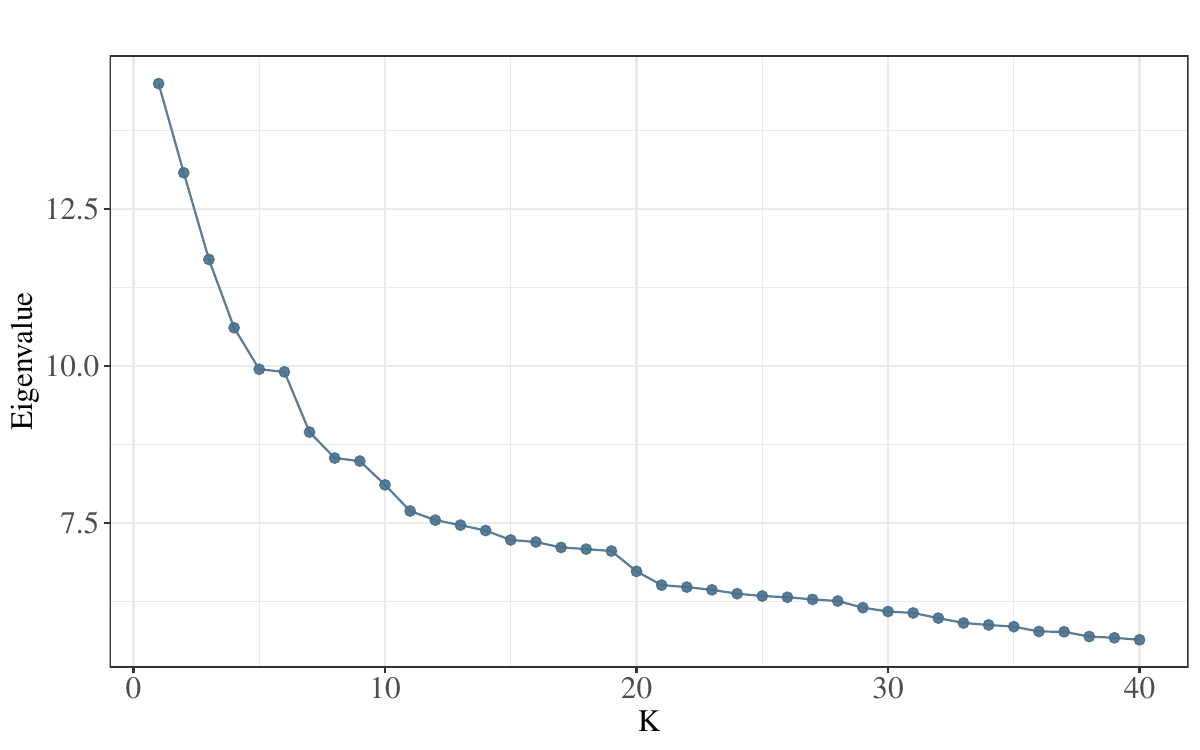}
    \caption{The scree plot of the primary collaboration network}
    \label{fig:scree_plot_collaboration}
\end{figure}

\begin{table}[!htp]
\scriptsize
\caption{Utilizing the information from three bipartite networks: author-paper, author-institution, and author-region networks to assist the community detection results of the primary network (collaboration network). } 
\newcommand{\tabincell}[2]{\begin{tabular}{@{}#1@{}}#2\end{tabular}}
\label{tab:BASIC_result}
\begin{center}
\begin{tabular}{ccccl}
\hline
    \textbf{Com} & \textbf{Size} & \textbf{Author} & \textbf{Citation Counts} & \textbf{Representative Keywords} \\
    \hline
    \multirow{5}{*}{1} & \multirow{5}{*}{148} & Balakrishnan, Narayanaswamy & 9.8k & \multirow{5}{*}{\tabincell{l}{Order statistics, Censored data, Survival analysis,\\ Maximum likelihood estimation, Birnbaum--Saunders\\ distribution}}\\
    ~ & ~ & Hothorn, Torsten & 5.8k & ~ \\
    ~ & ~ & Rao, J. N. K. & 5.1k & ~ \\
    ~ & ~ & Kundu, Debasis & 4.7k & ~ \\
    ~ & ~ & Cordeiro, Gauss Moutinho & 4.3k & ~ \\
    \hline

    \multirow{5}{*}{2} & \multirow{5}{*}{139} & Tibshirani, Robert & 52.0k & \multirow{5}{*}{\tabincell{l}{Survival analysis, Causal inference, Nonparametric\\ regression, LASSO, Functional data analysis}}\\
    ~ & ~ & Hastie, Trevor J. & 37.8k & ~ \\
    ~ & ~ & Friedman, Jerome H. & 36.0k & ~ \\
    ~ & ~ & Fine, Jason P. & 10.3k & ~ \\
    ~ & ~ & Wei, Lee-Jen & 10.1k & ~ \\
    \hline

    \multirow{5}{*}{3} & \multirow{5}{*}{91} & Bai, Zhidong & 3.0k & \multirow{5}{*}{\tabincell{l}{Covariance matrix, Dimension reduction, Random\\ matrix, Multivariate analysis, High-dimensional data}}\\
    ~ & ~ & Li, Wai-Keung & 2.6k & ~ \\
    ~ & ~ & Fang, Kaitai & 1.6k & ~ \\
    ~ & ~ & Peng, Heng & 1.5k & ~ \\
    ~ & ~ & Yin, Guosheng & 1.5k & ~ \\
    \hline

    \multirow{5}{*}{4} & \multirow{5}{*}{75} & Carlin, Bradley P. & 11.7k & \multirow{5}{*}{\tabincell{l}{Bayesian inference, Empirical Bayes, Central subspace,\\ Hierarchical models, Dimension reduction}}\\
    ~ & ~ & Gelfand, Alan E. & 11.7k & ~ \\
    ~ & ~ & Cook, R. Dennis & 5.8k & ~ \\
    ~ & ~ & Casella, George & 5.7k & ~ \\
    ~ & ~ & Wu, C. F. Jeff & 5.7k & ~ \\
    \hline

    \multirow{5}{*}{5} & \multirow{5}{*}{57} & Fan, Jianqing & 20.7k & \multirow{5}{*}{\tabincell{l}{Variable selection, LASSO, Model selection, Oracle\\ property, High-dimensional data, SCAD}}\\
    ~ & ~ & Li, Runze & 10.7k & ~ \\
    ~ & ~ & Zou, Hui & 9.0k & ~ \\
    ~ & ~ & Tsai, Chih-Ling & 7.6k & ~ \\
    ~ & ~ & Hornik, Kurt & 6.7k & ~ \\
    \hline

    \multirow{5}{*}{6} & \multirow{5}{*}{45} & Hall, Peter & 17.7k & \multirow{5}{*}{\tabincell{l}{Robustness, Bandwidth, Kernel methods, Small area\\ estimation, Nonparametric regression}}\\
    ~ & ~ & Zeileis, Achim & 8.1k & ~ \\
    ~ & ~ & Mukerjee, Rahul & 1.6k & ~ \\
    ~ & ~ & Cuevas, Antonio & 1.4k & ~ \\
    ~ & ~ & Basu, Analabha & 1.3k & ~ \\
    \hline

    \multirow{5}{*}{7} & \multirow{5}{*}{40} & Rousseeuw, Peter J. & 6.8k & \multirow{5}{*}{\tabincell{l}{Linear mixed model, Missing data, Longitudinal data,\\ Breakdown point, Random effects}}\\
    ~ & ~ & Kenward, Michael G. & 6.3k & ~ \\
    ~ & ~ & Molenberghs, Geert & 5.3k & ~ \\
    ~ & ~ & Croux, Christophe & 3.2k & ~ \\
    ~ & ~ & Verbeke, Geert & 3.1k & ~ \\
    \hline

    \multirow{5}{*}{8} & \multirow{5}{*}{37} & Ruppert, David & 6.7k & \multirow{5}{*}{\tabincell{l}{Functional data analysis, Semiparametric estimation,\\ Measurement error, Penalized splines, Smoothing}}\\
    ~ & ~ & Stefanski, Leonard A. & 2.8k & ~ \\
    ~ & ~ & Liang, Hua & 2.5k & ~ \\
    ~ & ~ & Crainiceanu, Ciprian M. & 2.3k & ~ \\
    ~ & ~ & Kneib, Thomas & 2.2k & ~ \\
    \hline

    \multirow{5}{*}{9} & \multirow{5}{*}{31} & Marron, James S. & 8.0k & \multirow{5}{*}{\tabincell{l}{Nonparametric regression, Kernel smoothing, Kernel\\ density estimation, Bandwidth selection, Density estimation}}\\
    ~ & ~ & Haerdle, Wolfgang Karl & 6.1k & ~ \\
    ~ & ~ & Wand, Matt P. & 5.7k & ~ \\
    ~ & ~ & Jones, M. C. & 5.5k & ~ \\
    ~ & ~ & Mammen, Enno & 4.6k & ~ \\
    \hline

    \multirow{5}{*}{10} & \multirow{5}{*}{29} & Ibrahim, Joseph G. & 5.5k & \multirow{5}{*}{\tabincell{l}{Missing data, Proportional hazards model, Longitudinal\\ data, Variable selection, Generalized estimating equations}}\\
    ~ & ~ & Lipsitz, Stuart R. & 3.6k & ~ \\
    ~ & ~ & Zeng, Donglin & 3.2k & ~ \\
    ~ & ~ & Ryan, Louise M. & 2.0k & ~ \\
    ~ & ~ & Zhu, Hongtu & 1.9k & ~ \\
    \hline

    \multirow{5}{*}{11} & \multirow{5}{*}{28} & Gijbels, Irene & 3.4k & \multirow{5}{*}{\tabincell{l}{Forward search, Nonparametric regression, Survival analysis,\\ Robustness, Bootstrap}}\\
    ~ & ~ & Hjort, Nils Lid & 2.6k & ~ \\
    ~ & ~ & Mardia, Kanti V. & 2.6k & ~ \\
    ~ & ~ & Morgan, Byron J. T. & 1.8k & ~ \\
    ~ & ~ & Atkinson, Anthony C. & 1.7k & ~ \\
    \hline

    \multirow{5}{*}{12} & \multirow{5}{*}{17} & Carroll, Raymond J. & 11.8k & \multirow{5}{*}{\tabincell{l}{Measurement error, Nonparametric regression, Dimension\\ reduction, Bayesian inference, Spatial statistics}}\\
    ~ & ~ & Smith, Adrian F. M. & 7.3k & ~ \\
    ~ & ~ & Zhu, Lixing & 5.6k & ~ \\
    ~ & ~ & Dette, Holger & 4.1k & ~ \\
    ~ & ~ & Genton, Marc G. & 4.0k & ~ \\
    \hline
\end{tabular}
\end{center}
\end{table}

Subsequently, using the author--paper, author--institution, and author--region networks as bipartite networks, we investigate the community structure of the primary network (the collaboration network) using the newly proposed BASIC method. Table~\ref{tab:BASIC_result} presents the five representative authors in each community, the community size, the total citation counts, and the top five keywords by frequency. Specifically, the representative authors in each community are selected by ranking authors according to their total citation counts of papers in our dateset according to {\it Web of Science} (up to 2022).
Communities are sorted by size in descending order. From Table~\ref{tab:BASIC_result}, we observe that the largest community contains 148 authors, whereas the smallest community contains 17 authors. Specifically, the top three authors and the fifth author in Community 2 are all affiliated with Harvard University and collaborate closely with one another. The fourth author, Fine, Jason P., is from the University of North Carolina. However, Wei, Lee-Jen, who is the doctoral advisor of Fine, Jason P., is from Harvard University. Therefore, it is reasonable to classify them into the same community. The authors in Community 3 are mainly affiliated with institutions in China, including Northeast Normal University, the Chinese Academy of Sciences, the University of Hong Kong, and Hong Kong Baptist University. Next, we focus on Community 5, which contains 57 authors. The top five authors in Community 5 mainly work in the area of high-dimensional statistics. Fan, Jianqing is the doctoral advisor of Li, Runze, and they have made significant contributions to penalized regression and screening methods. Most authors in Community 7 are from Belgium, including institutions such as KU Leuven, Hasselt University, and Ghent University. The top-ranked author is Rousseeuw, Peter J., a renowned statistician at KU Leuven whose research focuses on robust statistics and cluster analysis. Molenberghs, Geert and Croux, Christophe are among his doctoral students.

\subsection{Community Structure and Collaboration Patterns}

In this subsection, we select three representative communities to further analyze collaboration patterns among statisticians. These include the second largest community, i.e., Community 2, the primarily Chinese statisticians in Community 5, and Community 10, which consists of statisticians from different regions. Figure~\ref{fig:sub_com2} shows the visualization of the largest connected component in Community 2, with several representative nodes (authors) highlighted in dark blue.
We find that some authors play a ``bridging'' role in collaborations. For example, Fine, Jason P. (represented by the purple node in Figure~\ref{fig:sub_com2}) from the University of North Carolina at Chapel Hill acts as a bridge connecting two relatively closely connected groups of nodes, with relatively betweenness centrality 0.132. One group consists of his Ph.D. advisor, Wei, Lee-Jen, together with his peers, Cai, Tianxi and Lu, Tian, while the other group includes his colleague, Hudgens, Michael G., also from the University of North Carolina at Chapel Hill. They are all renowned statisticians in the field of biostatistics. Similarly, Cai, Tianwen Tony from the Wharton School of the University of Pennsylvania also serves as a bridge, connecting the group of colleagues at the University of Pennsylvania with the group that includes his sister, Cai, Tianxi, from Harvard University. This bridging role is even more pronounced with betweenness centrality 0.453. In addition, the professors in the group on the right have all studied or worked at the Harvard School of Public Health, including Betensky, Rebecca A., Robins, James M., and others. In summary, the authors in this community are all involved in the fields of biostatistics and public health statistics.

\begin{figure}[!ht]
    \centering
    \includegraphics[width=0.4\textwidth]{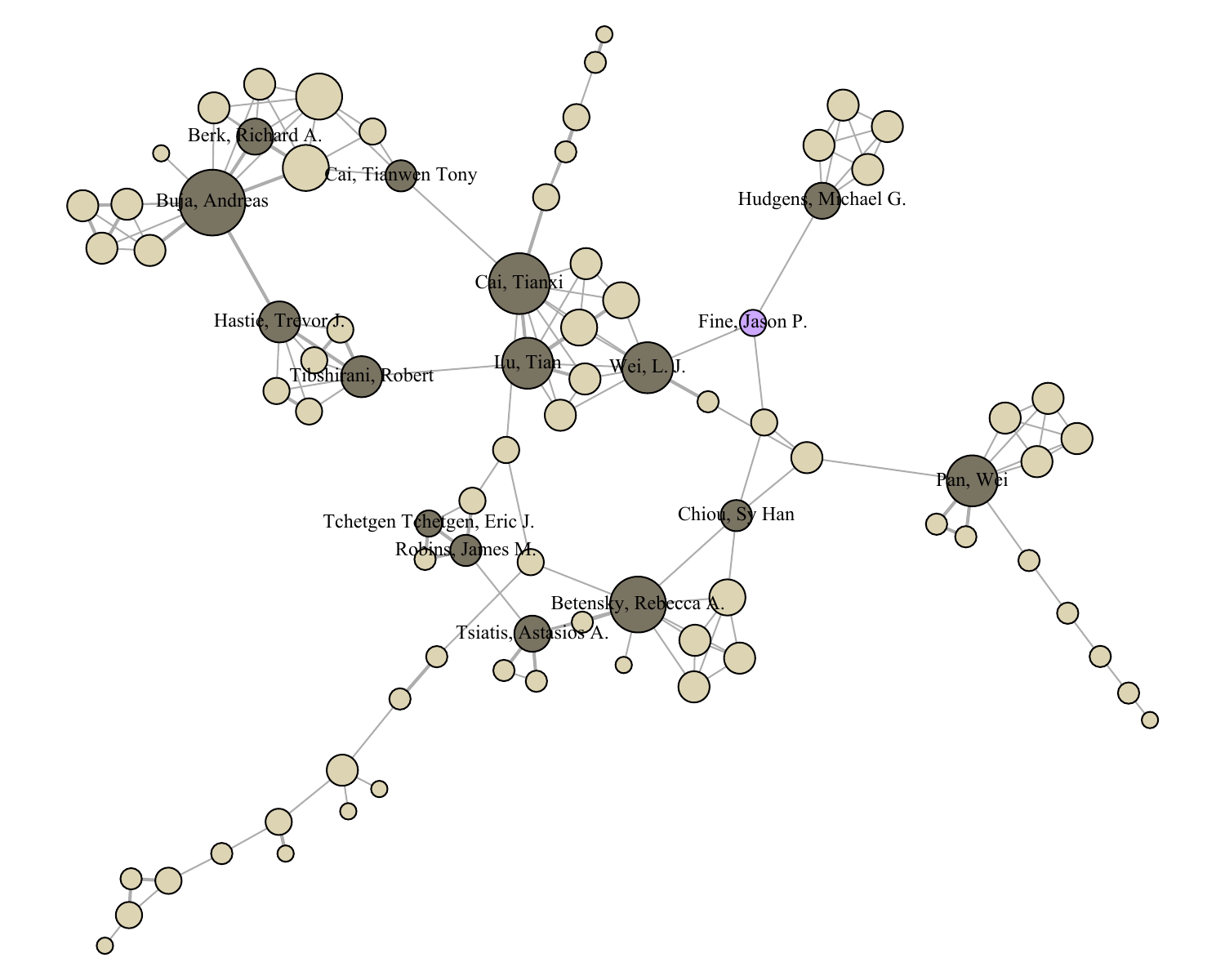}
    \caption{The largest connected component of Community 2 in the collaboration network.}
    \label{fig:sub_com2}
\end{figure}

Figure~\ref{fig:sub_com5} shows the largest connected component of Community 5 in the collaboration network. The representative nodes in this community include Fan, Jianqing from Princeton University, Li, Runze from Pennsylvania State University, Zou, Hui from the University of Minnesota, Tsai, Chih-Ling from the University of California, Davis, and others. In addition, we observe many interesting phenomena, which further validate the effectiveness of our method. For example, Fan, Jianqing, Li, Runze, and Zou, Hui form a loop, indicating that they collaborate very closely. Ma, Shuangge and Huang, Jian, along with others, are relatively closely connected, appearing to form a sub-community. Leng, Chenlei from the University of Warwick serves as a bridge in this community, connecting the group represented by Qin, Jing from the Hong Kong Polytechnic University. From the perspective of research directions, the authors in this community all specialize in high-dimensional statistics.

\begin{figure}[!ht]
    \centering
    \includegraphics[width=0.4\textwidth]{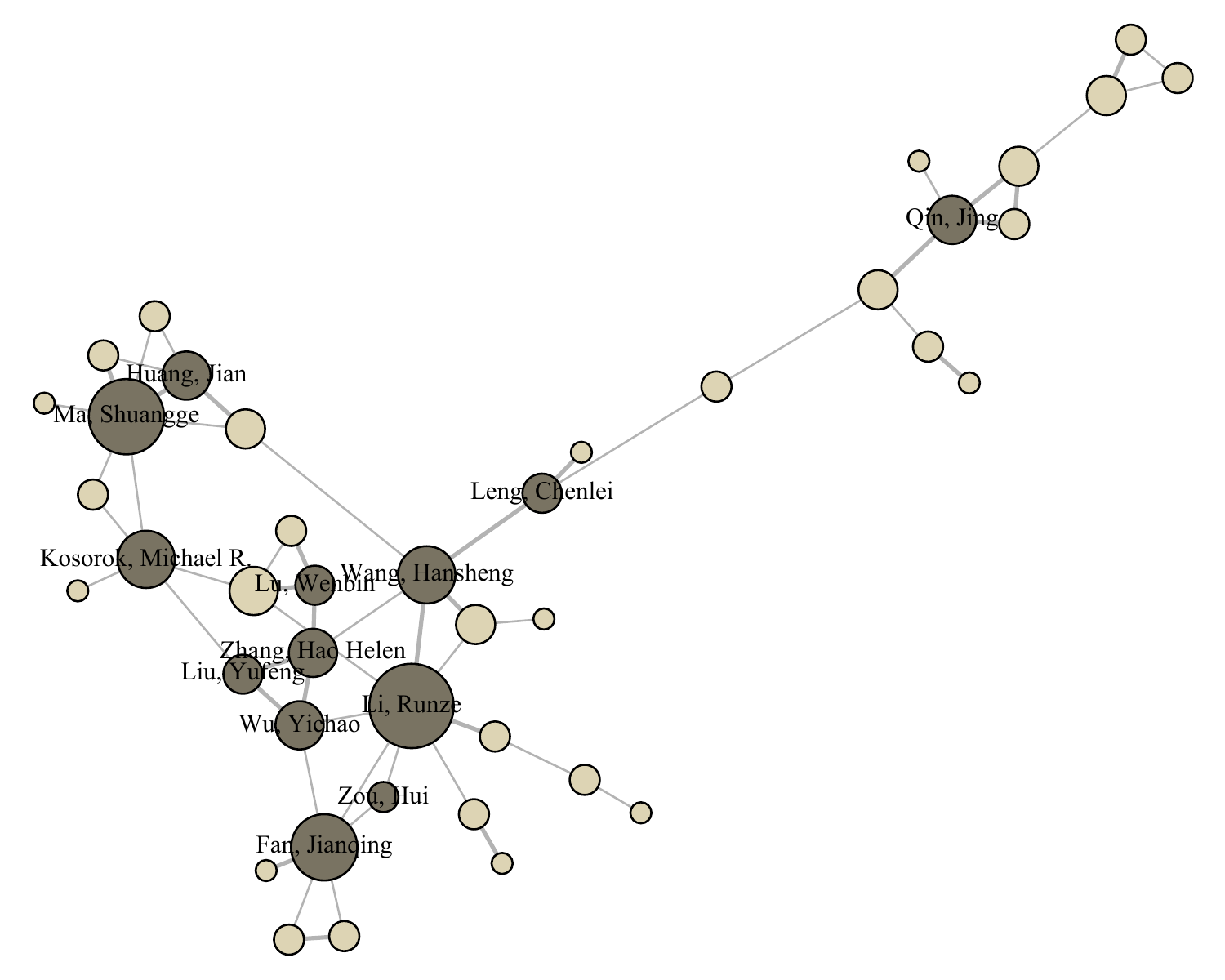}
    \caption{The largest connected component of Community 5 in the collaboration network.}
    \label{fig:sub_com5}
\end{figure}

The last community we discuss is Community 10. As shown in Figure~\ref{fig:sub_com10}, the visualization of Community 10 reveals that the nodes within this community are very closely connected. The density of this community is 21.2\%, which is much higher than that of Community 2 (5.5\%) and Community 5 (9.3\%). This suggests that collaboration among the authors in this community is much closer. The representative authors of this community, including Ibrahim, Joseph G., Zhu, Hongtu, and Styner, Martin A., are all from the University of North Carolina at Chapel Hill and are well-known statisticians. It is also interesting that both Ibrahim, Joseph G. and Zhu, Hongtu are the Ph.D. advisors of Shi, Xiaoyan. Additionally, there are many higher-order complete subnetwork structures in this community, where every pair of distinct nodes is connected by an edge. For example, a four-node complete subgraph is formed by Ibrahim, Joseph G., Sinha, Debajyoti, Lipsitz, Stuart R., and Fitzmaurice, Garrett A., which further indicates that this community is highly collaborative.

\begin{figure}[!ht]
    \centering
    \includegraphics[width=0.3\textwidth]{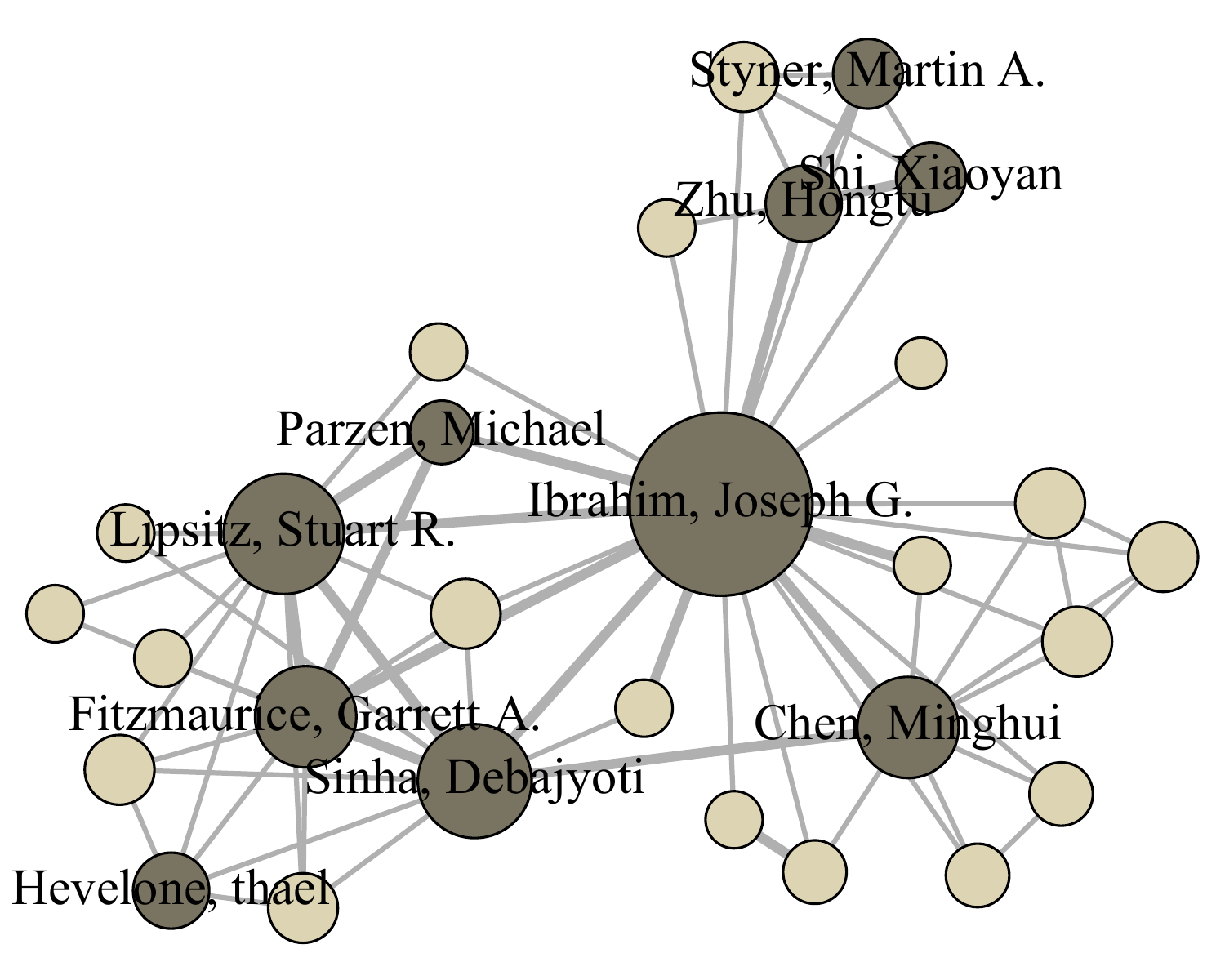}
    \caption{The largest connected component of Community 10 in the collaboration network.}
    \label{fig:sub_com10}
\end{figure}

\subsection{Dynamic Evolution of Communities}

To study the dynamic evolution of collaboration structure, we consider three nested time windows: 1981–2001, 1981–2011, and 1981–2021. For each time window, we construct the corresponding collaboration network and the three related bipartite networks. 
The main characteristics of the collaboration network and the three bipartite networks in the three time periods are summarized in Table~\ref{tab:dynamic_network_summary}. As expected, the network size increases substantially over time, while the density of the collaboration network decreases, indicating that the growth of collaborations is accompanied by an increasingly sparse structure.

\begin{table}[!ht]
\centering
\caption{Summary of the collaboration network and three bipartite networks across three nested time windows}
\label{tab:dynamic_network_summary}
\begin{tabular}{lcccccc}
\hline
Time window & Nodes & Edges & Density(\%)  & Author--Paper & Author--Institution & Author--Region \\
\hline
1981--2001 & 139 & 256  & 2.669 & $139 \times 2810$  & $139 \times 80$  & $139 \times 19$ \\
1981--2011 & 441 & 1038 & 1.069 & $441 \times 3906$  & $441 \times 197$ & $441 \times 32$ \\
1981--2021 & 737 & 2453 & 0.904 & $737 \times 17722$ & $737 \times 295$ & $737 \times 44$ \\
\hline
\end{tabular}
\end{table}

We next examine how community structures evolve across the three nested time windows. The community detection result for the 1981--2021 window is presented in Section~4.1. For 1981--2001 and 1981--2011, we apply the same BASIC procedure and determine the number of communities based on the scree plot and the ECV criterion, which suggest 6 and 9 communities, respectively.
We match communities across adjacent time windows based on node overlap.
For each community in an earlier window, we calculate its inheritance proportion in the next window, defined as the proportion of its nodes that remain together in the best-matched community in the following period.
A higher inheritance proportion indicates a more stable community over time, meaning that most existing members are preserved while new members may be added.
Based on this measure, we classify community transitions into two types.
If the inheritance proportion is greater than 50\%, we define the community as \emph{survive}; otherwise, we define it as \emph{split}.
Figure~\ref{fig:dynamicnature} illustrates the evolution of communities across the three time windows.

\begin{figure}
    \centering
    \includegraphics[width=1\linewidth]{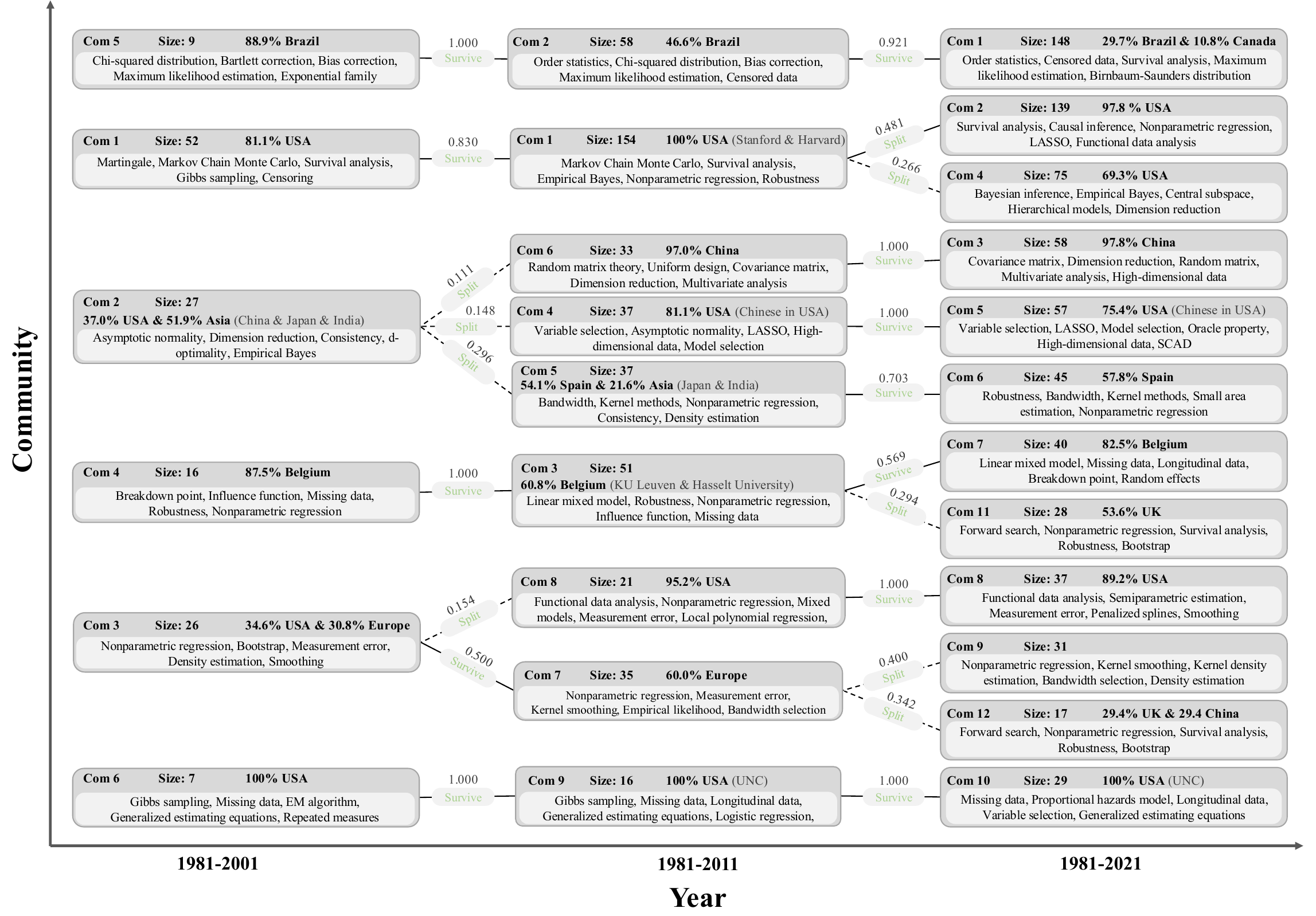}
    \caption{Evolution of communities across the three nested time windows (1981--2001, 1981--2011, and 1981--2021). Communities are matched based on node overlap. The values on the connecting edges indicate inheritance proportions, and each transition is labeled as either \emph{Survive} or \emph{Split} according to the 50\% threshold.}
    \label{fig:dynamicnature}
\end{figure}

The boxes in Figure~\ref{fig:dynamicnature} represent the detected communities. For each box, the upper panel reports the community size together with geographic information, such as the main country composition and representative institutions.
The lower panel lists representative keywords selected by within-community frequency. The results suggest that the major research directions in the field exhibit a mixed pattern of continuity and reorganization.
Several communities remain highly stable over time, with inheritance proportions close to or equal to 1, indicating that their core members are largely preserved as the network expands.
Examples include the communities centered on Bayesian computation and survival analysis (Com 1 in 1981--2001), robustness and missing data (Com 4 in 1981--2001), and Gibbs sampling and longitudinal modeling (Com 6 in 1981--2001), all of which survive into later periods with clear trajectories.
At the same time, some early communities gradually split into more specialized groups as the field develops. A clear example is Com 2 in 1981--2001, which is initially characterized by asymptotic normality, dimension reduction, consistency, and empirical Bayes.
This community does not remain as a single coherent group in the next period, but instead branches into several more focused communities related to high-dimensional variable selection, random matrix theory and multivariate analysis, and kernel-based nonparametric methods.
This pattern suggests that a broad methodological cluster in the early stage later differentiates into more specialized research directions.

In addition, the regional and institutional information provides a useful view of the community structure in the collaboration network. Some communities show a clear concentration in specific regional or institutional groups, suggesting that certain collaboration patterns are partly shaped by regional or institutional proximity. For example, several communities are dominated by scholars from Belgium or China, reflecting a certain degree of regional concentration in the collaboration network. In some cases, the institutional information also suggests closely connected collaboration circles, such as communities with a high proportion of Chinese scholars affiliated with U.S. institutions. These patterns suggest that the communities in the collaboration network are shaped not only by methodological similarity, but also by geographic proximity, institutional ties, and cross-national academic mobility.

The temporal analysis reveals both persistent collaboration cores and progressive specialization. For example, the community centered on missing data, generalized estimating equations, and longitudinal analysis has inheritance proportions of 1.000 in both transitions and remains institutionally concentrated around UNC. The Brazil-centered community on classical inference and survival analysis also remains highly persistent, with inheritance proportions of 1.000 and 0.921. In contrast, the broad 1981--2001 community associated with asymptotic theory, dimension reduction, empirical Bayes, and optimal design has no dominant successor retaining at least half of its members. It branches into more specialized communities associated with random matrix theory and multivariate analysis, high-dimensional variable selection, and kernel-based nonparametric methods; these descendants subsequently exhibit stronger persistence. Another community combining MCMC, survival analysis, and Gibbs sampling initially survives with an inheritance proportion of 0.830 but later separates mainly into a survival/causal-inference group and a Bayesian/empirical-Bayes group. These trajectories suggest that stability and splitting are sequential rather than mutually exclusive features. A collaboration group may retain a stable core during one stage and differentiate as the field develops.

The increase from 6 to 9 and then 12 detected communities is consistent with increasing specialization, but it should be interpreted together with the substantial growth of the network from 139 to 441 and then 737 authors and the decrease in density from 2.669\% to 1.069\% and 0.904\%. Most importantly, the communities obtained from the full 1981--2021 network should be interpreted as a retrospective, time-aggregated map of long-run collaboration structure. They summarize relationships accumulated over the observation period and identify major long-run collaboration basins, but they do not imply that each community existed in an unchanged form for forty years or that individual authors had time-invariant memberships. The temporal trajectories in Figure~\ref{fig:dynamicnature} provide historical context for this aggregate partition, that is, some full-period communities descend from persistent collaboration cores, while others reflect the later specialization of broader earlier groups.

\subsection{Individual Contribution of Each Bipartite Network}

In this subsection, we investigate the contribution of each bipartite network to the community structure identified by BASIC. We refer to the model that uses the collaboration network together with all three bipartite networks as the \emph{full model}. To evaluate the role of each bipartite network, we fit BASIC using the collaboration network together with one bipartite network at a time, and compare the resulting partition with that from the full model. Since the primary network under consideration remains the collaboration network, we fix the number of communities at 12 throughout. This ensures that any change in the detected community structure is attributable to the inclusion of different bipartite networks, rather than to variation in the number of communities.

We first compare the overall agreement between the partitions using the concordance rate, defined as the proportion of nodes assigned to the same community after optimal label permutation. The concordance rates are 46.0\%, 34.7\%, and 45.2\% for the author--paper, author--institution, and author--region networks, respectively. A higher concordance rate indicates that the corresponding bipartite network preserves more of the community structure identified by the full model. These results show that the author--paper and author--region networks capture a larger portion of the full-model structure, while the author--institution network alone is relatively less informative. Table~\ref{tab:single_bipartite_overlap} reports the community-level correspondence between the full model and the three single-bipartite-network variants. The community IDs are relabeled in descending order of community size. For each community in the full model, the table shows the matched community under the author–paper, author–institution, and author–region settings, together with community size and the number of overlapping nodes. This provides a more detailed view of how well each bipartite network preserves the community structure.

\begin{table}[!ht]
\centering
\scriptsize
\caption{Community correspondence, sizes, and overlapping node counts between the full BASIC model and the single-bipartite-network variants.}
\label{tab:single_bipartite_overlap}
\begin{tabular}{cc|ccc|ccc|ccc}
\hline
\multicolumn{2}{c|}{\textbf{Full model}}
& \multicolumn{3}{c|}{\textbf{Author--Paper}}
& \multicolumn{3}{c|}{\textbf{Author--Institution}}
& \multicolumn{3}{c}{\textbf{Author--Region}} \\
\textbf{ID} & \textbf{Size}
& \textbf{ID} & \textbf{Size} & \textbf{Overlap}
& \textbf{ID} & \textbf{Size} & \textbf{Overlap}
& \textbf{ID} & \textbf{Size} & \textbf{Overlap} \\
\hline
Com 1  & 148 & Com P-1  & 141 & 106 & Com I-6  & 62  & 43 & Com R-6  & 31  & 31  \\
Com 2  & 139 & Com P-3  & 66  & 25  & Com I-2  & 88  & 44 & Com R-1  & 344 & 139 \\
Com 3  & 91  & Com P-2  & 84  & 41  & Com I-5  & 62  & 41 & Com R-2  & 80  & 65  \\
Com 4  & 75  & Com P-8  & 49  & 25  & Com I-3  & 86  & 34 & Com R-7  & 27  & 0   \\
Com 5  & 57  & Com P-5  & 63  & 18  & Com I-9  & 34  & 7  & Com R-9  & 24  & 2   \\
Com 6  & 45  & Com P-4  & 66  & 23  & Com I-8  & 38  & 10 & Com R-4  & 44  & 27  \\
Com 7  & 40  & Com P-7  & 50  & 38  & Com I-1  & 178 & 38 & Com R-5  & 38  & 36  \\
Com 8  & 37  & Com P-6  & 60  & 21  & Com I-10 & 26  & 16 & Com R-8  & 26  & 0   \\
Com 9  & 31  & Com P-12 & 30  & 11  & Com I-7  & 57  & 9  & Com R-3  & 66  & 17  \\
Com 10 & 29  & Com P-9  & 45  & 17  & Com I-11 & 19  & 0  & Com R-11 & 20  & 0   \\
Com 11 & 28  & Com P-11 & 38  & 8   & Com I-4  & 69  & 14 & Com R-12 & 17  & 16  \\
Com 12 & 17  & Com P-10 & 45  & 6   & Com I-12 & 18  & 0  & Com R-10 & 20  & 0   \\
\hline
\end{tabular}
\end{table}

Several patterns in Table~\ref{tab:single_bipartite_overlap} are particularly noteworthy. First, among the three single-network variants, the author--paper network preserves the full BASIC partition most effectively. For example, the largest BASIC community (Community 1, 148 nodes) is matched to an author--paper community of comparable size (Community P-1, 141 nodes), with 106 overlapping nodes, indicating substantial structural consistency. Similarly, Community 7 is almost perfectly preserved, with 38 of its 40 nodes overlapping with its matched author--paper Community P-7.

Second, the author--institution network captures meaningful institution-based structure. In particular, some institution-based communities correspond to recognizable organizational systems. For example, Community I-12 is largely composed of University of California campuses, including UCLA, UC Irvine, and UC Davis, indicating that the institution network can effectively recover system-level institution structure. Community I-5 is dominated by universities in China, suggesting that institutional information can also reflect strong national academic clustering. At the same time, not all institution-based communities are institutional homogeneous: Community I-9 contains a mixture of Hong Kong universities and U.S. institutions, illustrating that institution information may also capture cross-regional collaboration patterns. 

Last, the author--region network produces the coarsest partition and naturally emphasizes geographic information. In particular, several region-based communities are strongly dominated by a single country, such as Community R-1, where 305 of 344 nodes (88.7\%) are from the United States, Community R-2, where 76 of 80 nodes (95.0\%) are from China, and Community R-5, where 34 of 38 nodes (89.5\%) are from Belgium. Therefore, if the goal is to identify communities with clear regional tendencies, the author--region network can be a highly informative choice. By contrast, if the goal is to recover finer collaboration-driven structures, region information alone may be too coarse.

\subsection{Sensitivity to the construction of the collaboration network}

In the main analysis, we construct the binary collaboration network by connecting two authors if they coauthor at least two papers. This construction uses collaboration frequency to distinguish repeated coauthorship from isolated one-time interactions. This construction is consistent with the binary network setting of the DCBM and also limits the influence of a small number of exceptionally prolific author pairs. To examine the sensitivity of the empirical findings to the treatment of collaboration intensity, we additionally apply the computational procedure to the weighted collaboration network, whose edge weights are the observed numbers of coauthored papers. This analysis serves as an empirical robustness check complementary to the binary-network analysis. We retain the same 737 core authors and fix the number of communities at $K=12$, as in the main analysis. This weighted-input analysis is intended as an empirical sensitivity check rather than an extension of the theoretical framework to a weighted DCBM. Figure~\ref{fig:weighted_sensitivity1} shows the communities detected by BASIC using the weighted input together with the author--institution and author--region bipartite networks. 

\begin{figure}[!ht]
    \centering
    \includegraphics[width=0.4\linewidth]{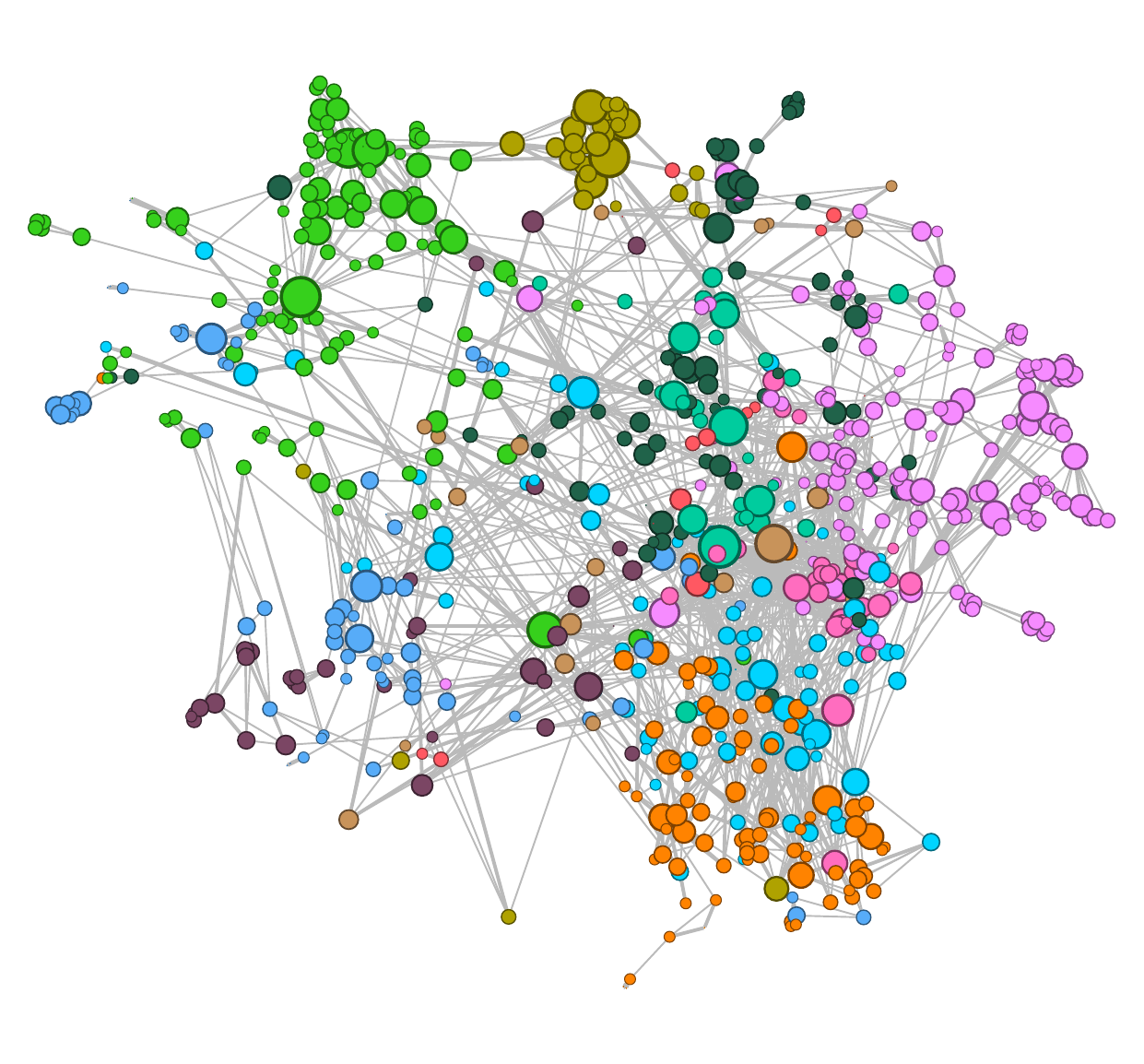}
    \caption{{Community detection results obtained by BASIC using the weighted collaboration network as the primary input.}}
    \label{fig:weighted_sensitivity1}
\end{figure}

Figure~\ref{fig:weighted_sensitivity} compares the resulting partitions with that obtained in the main analysis. Rows represent communities from the main analysis, and columns represent communities from the sensitivity analysis. The sensitivity communities are reordered using optimal one-to-one matching. Each cell reports the number of shared authors, with the corresponding percentage of the row community shown in parentheses. 
The two partitions have a concordance rate of 78.3\%, corresponding to 577 of the 737 authors. Ten of the twelve communities identified in the main analysis have a clear counterpart in the sensitivity analysis, with their dominant correspondence proportions ranging from 73.0\% to 95.7\%. Communities 2 and 5 are particularly stable, retaining 95.7\% and 91.2\% of their members, respectively. 

\begin{figure}[!ht]
    \centering
    \includegraphics[width=0.58\linewidth]{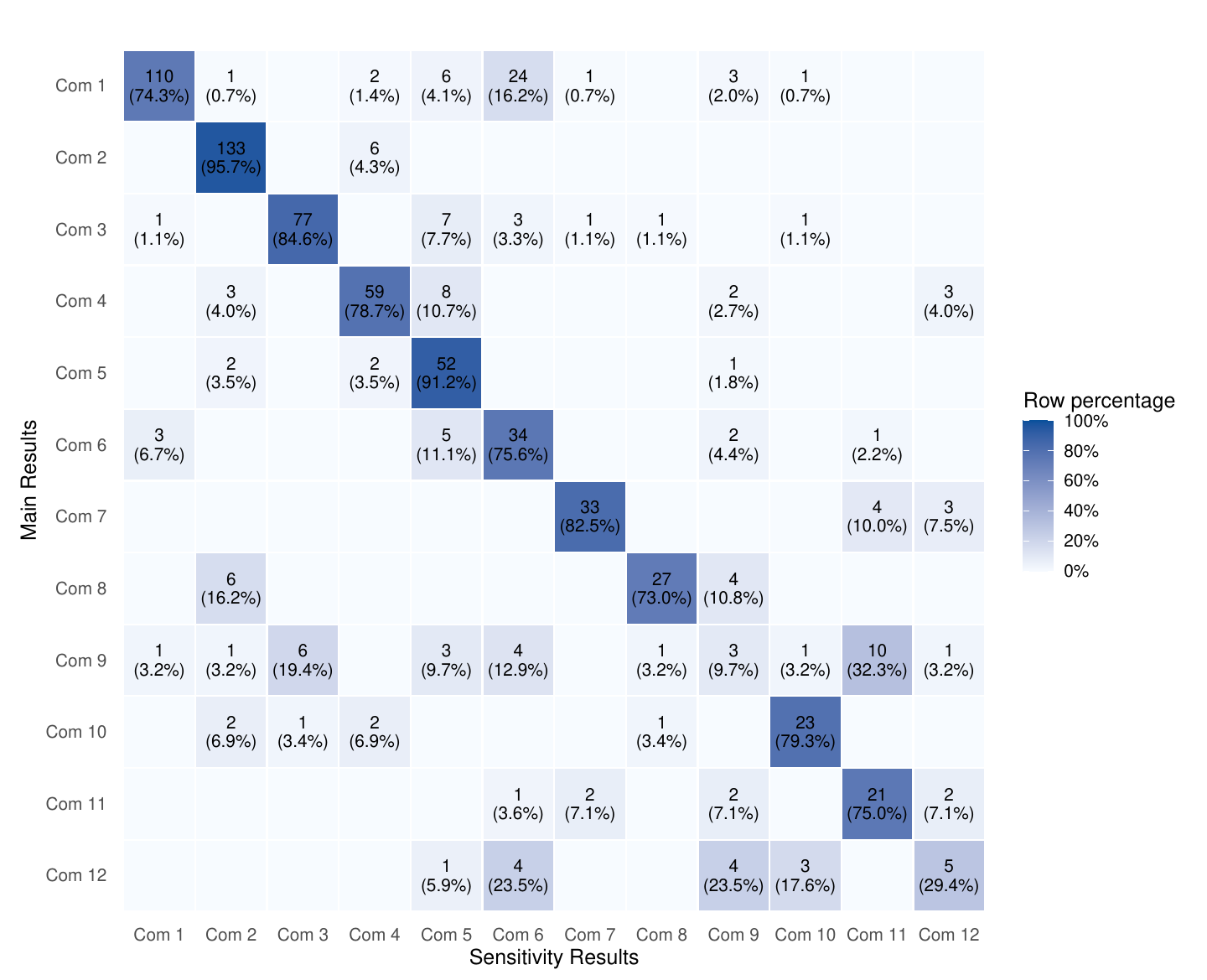}
    \caption{{Heatmap of community assignment correspondence between the main and sensitivity results. Rows represent communities from the main analysis, and columns represent communities from the sensitivity analysis. Each cell reports the number of shared authors, with the corresponding percentage of the row community shown in parentheses. Cell colors are normalized within rows, and the sensitivity-analysis communities are reordered using optimal one-to-one matching.}}
    \label{fig:weighted_sensitivity}
\end{figure}

The most pronounced differences occur for Communities 9 and 12, which contain 31 and 17 authors, respectively, and are among the smaller communities identified in the main analysis. Their members are distributed across several communities in the sensitivity analysis. This pattern is consistent with both their relatively small sizes and the fact that the weighted representation restores one-time coauthorship excluded by the threshold while giving more frequent collaborations proportionally greater influence. Smaller or less sharply separated communities may therefore be more sensitive to a limited number of strong coauthorship. Overall, the major community correspondences remain clearly recognizable, while the local differences reveal how collaboration intensity affects individual assignments. These findings provide empirical evidence that the main structural conclusions are not solely driven by the binary construction of the primary collaboration network. Because the proposed method is theoretically developed for binary networks, the binary construction remains the model-aligned specification for our main analysis, whereas the weighted-input analysis serves as a complementary empirical sensitivity check.

\subsection{Comparison with Other Methods}

To demonstrate the advantages of incorporating the bipartite network information, we compare the proposed BASIC method with two representative baseline methods, SCORE \citep{jin2015fast} and SCORE+ \citep{Jin_2021}. The SCORE method conducts community detection based solely on the primary network. SCORE+ is an enhanced version of SCORE that further incorporates pre-PCA normalization, Laplacian transformation of the adjacency matrix, and an additional eigenvector in the clustering step. For each of these two methods, we consider both the standard implementation applied directly to collaboration network, and a weighted variant constructed from the author–paper bipartite network.

Figure~\ref{fig:comparison} shows the different community detection results by BASIC, SCORE, SCORE+, SCORE (weighted), and SCORE+ (weighted) respectively. Nodes in different communities are assigned different colors. Overall, BASIC, SCORE, and SCORE (weighted) produce more balanced community size distributions, whereas the largest communities identified by SCORE+ and SCORE+ (weighted) contain 68.7\% and 69.7\% of all nodes, respectively. We find that SCORE+ is not suitable for our collaboration network. This is because SCORE+ requires that the $K$-th eigenvalue is very close to the $(K+1)$-th eigenvalue, while there is a relatively large gap between the $(K+1)$-th and $(K+2)$-th eigenvalues. However, as shown in Table~\ref{tab:ratio_k}, in the collaboration network studied in this paper, not only the $K$-th and $(K+1)$-th eigenvalues but even the $(K+2)$-th eigenvalue are very close, making it difficult to argue that considering one additional eigenvector would bring a significant improvement. Therefore, the SCORE+ method does not apply well to our collaboration network. By incorporating additional structural information into the collaboration network, we are able to achieve better results than SCORE+.

\begin{table}[!ht]
\caption{Variation of $1 - {\hat{\lambda}_{K+1}}/{\hat{\lambda}_K}$in our collaboration network as $K$ varies from 1 to 16. Note that the number of communities in our collaboration network is 12.}
\label{tab:ratio_k}
\begin{center}
\begin{tabular}{cccccccccccccccc}
    \hline
    \textbf{$K$} &  1 &  2 &  3 &  4 &  5 &  6 &  7 &  8 \\
    \hline
    \textbf{$1 - {\hat{\lambda}_{K+1}}/{\hat{\lambda}_K}$} & 0.098 &  0.105 &  0.093 & 0.062 & 0.004& 0.097& 0.046& 0.006 \\
    \hline
    \hline
    \textbf{$K$} & 9 &  10 &  11 &  \textbf{12} &  13 &  14 &  15 & 16\\
    \hline
    \textbf{$1 - {\hat{\lambda}_{K+1}}/{\hat{\lambda}_K}$} & 0.045& 0.051& 0.019 & \textbf{0.011} & 0.012 &0.020 & 0.004 & 0.012\\
    \hline
    \end{tabular}
    \end{center}
\end{table}

\begin{figure}[!ht]
    \centering
\includegraphics[width=0.95\textwidth]{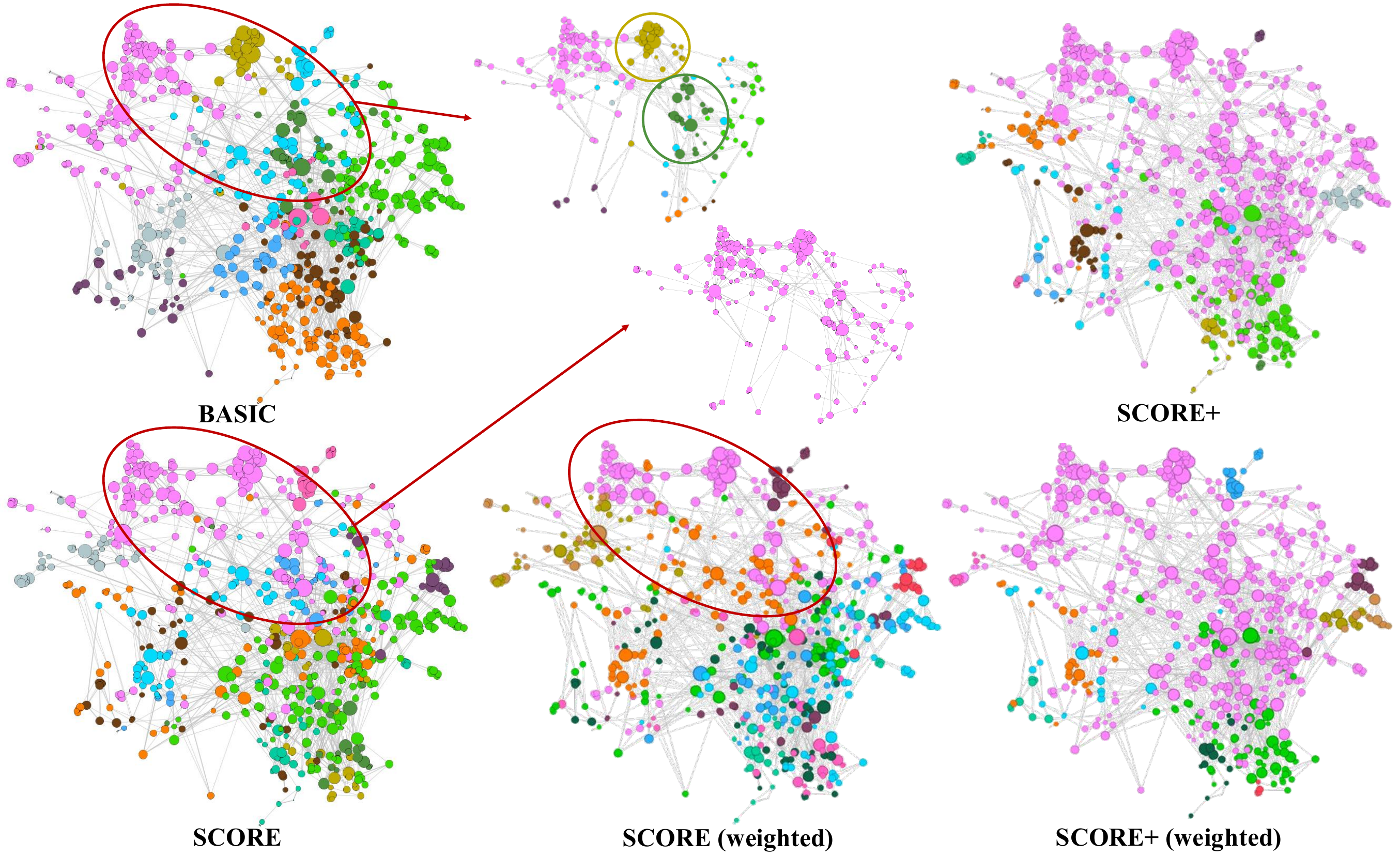}
    \caption{Top Left: Community detection results using the BASIC method; Top Middle: Community detection results for a subset of selected nodes obtained by the BASIC and SCORE methods; Top Right: Community detection results using the SCORE+ method; Bottom Left: Community detection results using the SCORE method; Bottom Middle: Community detection results using the weighted variant of SCORE; Bottom Right: Community detection results using the weighted variant of SCORE+. Nodes belonging to different communities are shown in different colors.}
    \label{fig:comparison}
\end{figure}

In comparing BASIC with SCORE and weighted SCORE, we find that SCORE assigns 55\% of the top 20 most cited authors to the same community, and weighted SCORE shows a similar pattern. These results suggest that both methods tend to over-concentrate highly influential authors into a single community. To provide a clearer comparison between BASIC and SCORE, the top-middle subplot shows the community structure of the selected nodes. Visually, these nodes exhibit distinct community structures, but the SCORE method groups all of them into the same community. In contrast, BASIC divides them into three main communities. Specifically, the yellow-node community within the yellow circle in Figure~\ref{fig:comparison} consists mainly of authors from KU Leuven and Hasselt University in Belgium. The community corresponding to the green nodes within the green circle is the highly connected Community 10 discussed above. In addition, SCORE assigns Fine, Jason P. and Wei, Lee-Jen to two different communities. However, Wei, Lee-Jen is the doctoral advisor of Fine, Jason P., and they collaborate closely. In contrast, BASIC assigns them to the same community and groups them together with authors from Harvard University.

In addition, we compare BASIC with representative methods for multilayer network community detection and community detection with covariates, including the covariate-assisted multi-layer stochastic block model (CAMSBM) \citep{xu2023covariate}, the least-squares estimator (LSE) \citep{lei2020consistent}, and mean adjacency spectral embedding (MAS) \citep{han2015consistent}. Specifically, to make these methods applicable in our bipartite setting, we construct a series of author--author adjacency matrices from the bipartite network and treat them as multiple network layers for the multilayer baselines, while using the corresponding derived bipartite information as auxiliary node-level covariates for the covariate-assisted baseline. Figure \ref{fig:multicomparsion} shows the community detection results using CAMSBM, LSE, and MAS, respectively. From an overall perspective, LSE yields a highly unbalanced community structure, with the two largest communities accounting for 51.4\% and 39.8\% of all nodes, respectively, indicating that most nodes are concentrated in only two communities. Therefore, in what follows, we focus on comparing the community labels produced by BASIC with those from MAS and CAMSBM.

\begin{figure}[!ht]
    \centering
    \includegraphics[width=0.95\linewidth]{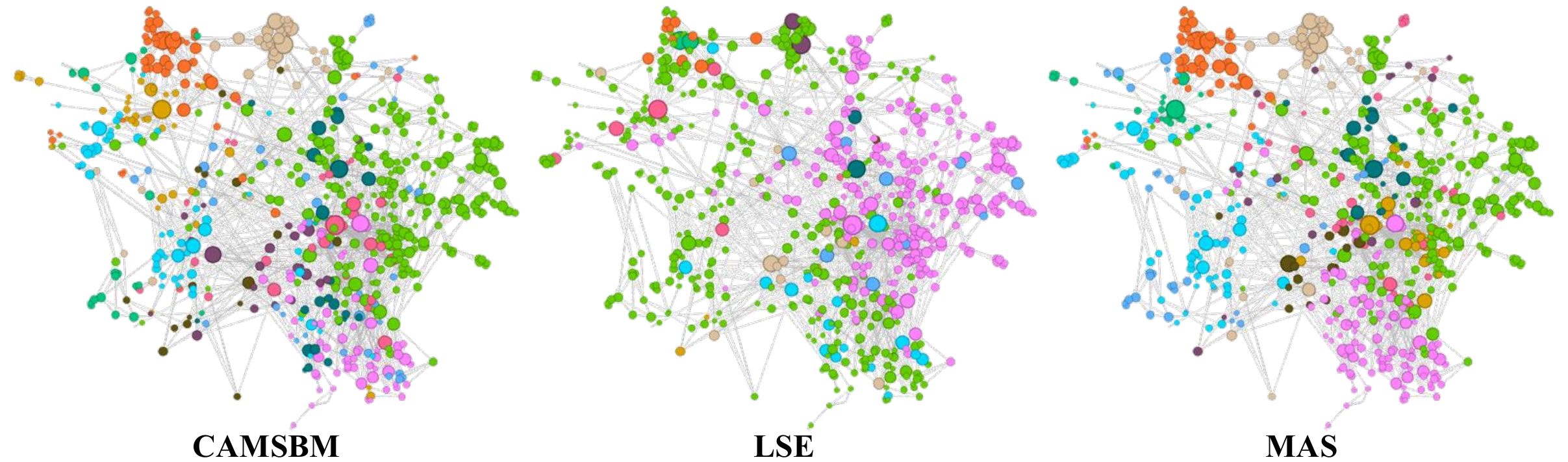}
    \caption{Left: Community detection results using the CAMSBM method; Middle: Community detection results using the LSE method; Top Right: Community detection results using the MAS method.}
    \label{fig:multicomparsion}
\end{figure}

We begin by quantifying the concordance rates between BASIC and the two baseline methods, MAS and CAMSBM. The concordance rates are 60.5\% for MAS and 49.8\% for CAMSBM, respectively. We therefore focus our subsequent analysis on nodes whose community assignments differ across methods. In particular, we place greater emphasis on highly cited authors, since their community memberships are more likely to be representative and substantively informative. For MAS, the community assignments of several highly influential authors appear less plausible from the perspective of observed collaboration patterns. In particular, within Community 5 identified by BASIC, MAS assigns Fan, Jianqing---arguably the most representative member of this community---to the same community as many European scholars, while separating him from his close collaborators, including Li, Runze and Zou, Hui. Next, the most prominent issue with CAMSBM is that its detected communities are almost entirely organized by region. For example, U.S.-based scholars are largely grouped into one community, China-based scholars into another, Belgium-based scholars into a third, and U.K.-based scholars into yet another. Such a partition largely overlooks institutional affiliation and collaboration structure, suggesting that region information dominates the inferred communities.

\section{Simulation}\label{Simulation}
In this section, we assess the performance of the proposed BASIC method under the DCBM for the primary network and the BiDCBM for bipartite networks, under both weak and strong signal conditions of the primary network. We investigate how bipartite networks with varying signal strengths can indeed enhance community detection, and verify that the clustering performance is not degraded even if the added bipartite information is weak. We consider various combinations of node sizes and the number of communities, addressing both balanced and imbalanced community structures. To evaluate the clustering accuracy of BASIC, we calculate the Adjusted Rand Index (ARI) \citep{hubert1985comparing} that reflects the consistency between the clustering results and the inherent true labels, ranging from -1 to 1. A higher ARI value indicates higher clustering accuracy. The SCORE method \citep{jin2015fast} applied to the primary network is treated as baseline.

\subsection{Simulation Setup}

We generate the DCBM signal matrices $\Omega^{(0)}$ and $\Omega^{(q)}$, $q=1,\ldots, Q$, for the primary and bipartite networks, respectively, following \cite{li2020network}. We take three combinations of $\{n,m,K\}$: $\{600, 300, 3\}$, $\{600, 300, 5\}$, and $\{1200, 600, 5\}$. For the community structure of our primary interest, both balanced and imbalanced community sizes are considered. In balanced cases, all communities have the same sizes $n/K$. In imbalanced cases, the community sizes of the primary nodes are respectively set to $\{100, 200, 300\}$, $\{50, 100, 100, 150, 200\}$ and $\{100, 200, 200, 300, 400\}$. Then assign the node membership labels $\boldsymbol{l}=(l_1,\ldots, l_n)^\top$ and $\boldsymbol{r}^{(q)}=(r_1^{(q)},\ldots, r_m^{(q)})^\top$, without loss of generality, in a sequential manner. Taking the balanced case for instance,
$$\boldsymbol{l}=(\underbrace{1,\ldots,1}_{n/K },\underbrace{2,\ldots,2}_{n/K },\ldots, \underbrace{K,\ldots,K}_{n/K })^\top.$$
Given $\boldsymbol{l}$, further define the community membership matrix $\mathbf{X}\in \mathbb{R}^{n\times K}$ for the primary network, where $\mathbf{X}(i,k)=\mathbb{I}$($l_i=k$) for $i=1,\ldots,n$ and $k=1,\dots,K$.
In addition, take $\boldsymbol{\Pi}=(1-\beta) \mathbf{I}_{K}+\beta\mathbf{1}_{K} \mathbf{1}_{K}^{\top}$, where $\mathbf{I}_{K}$ is a $K \times K$ identity matrix and $\mathbf{1}_{K}$ denotes a column vector of length $K$ with every entry being 1, and $\beta$ represents the common out-in ratio, i.e., the ratio of between-block and within-block probability of edges. The values of $\beta$ will be specified later in different scenarios. As $\beta$ increases, the communities become less distinguishable, resulting in weak-signal networks. Last, we draw a vector of node degree parameter $\mathbf{d}=(d_1,\ldots, d_n)^\top$ from a power-law distribution with lower bound 1 and scaling parameter 5. When normalized to the range between 0 and 1, $\mathbf{d}$ corresponds to the vector of degree heterogeneity parameter $\boldsymbol{\theta}$ in Section~A of the Supplementary Material \citep{gao2026basic_supp}.
Given all the above quantities, we can define the mean matrix $\boldsymbol{\Omega}^{(0)} \propto \diag(\mathbf{d}) \mathbf{X} \boldsymbol{\Pi} \mathbf{X}^{\top}\diag(\mathbf{d})^{\top}$ according to the DCBM, so that the probability of an edge between nodes only depends on the community structure $\mathbf{X}$ and the node degree parameter $\mathbf{d}$. We can specify a normalizing factor when generating $\boldsymbol{\Omega}^{(0)}$ to adjust the average node degree in the network, aiming to prevent the network from being too dense or sparse. In our simulation, we set the average degree to be 40. In the same fashion, we can generate the mean matrices $\boldsymbol{\Omega}^{(q)}$, $q=1,\ldots, Q$ for the bipartite networks. Here we take $Q=3$ or $5$ in different case. All simulation results are based on 200 replications.

\subsection{Simulation Results}
We evaluate the performance of BASIC under various signal conditions of primary and bipartite networks.
In the first scenario, we consider the weak-signal primary network, with the out-in ratio $\beta$ of the primary network set to $0.5$ \citep{li2020network}. Recall that a larger $\beta$ leads to a weaker signal. Then, we vary the out-in ratios of the 5 bipartite networks in the following four cases:

\begin{figure}[!t]
    \centering
    \includegraphics[width=0.85\textwidth]{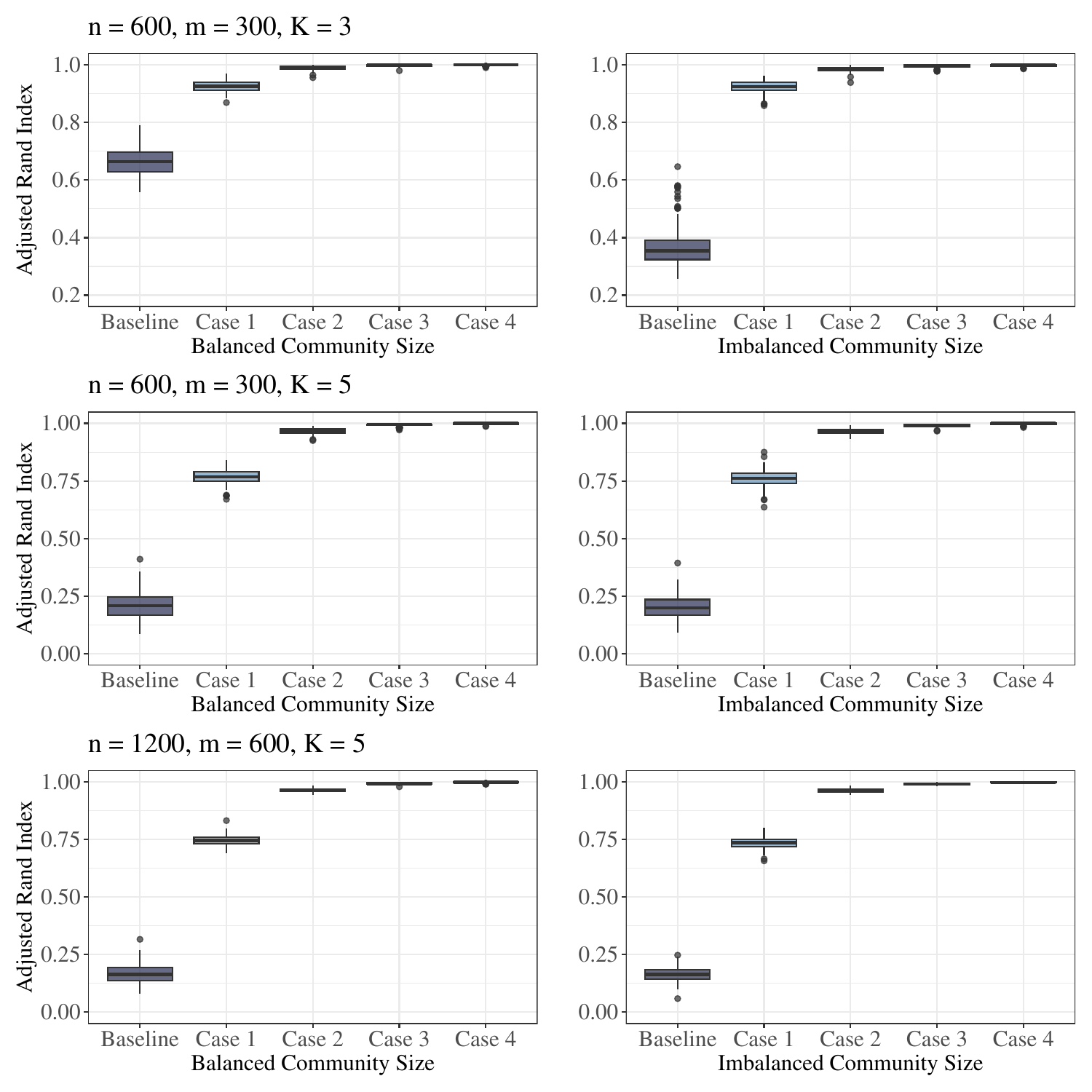}
    \caption{ARI for weak-signal primary networks. The baseline refers to the results obtained by using only the primary network. Cases 1, 2, 3, and 4 correspond to using 0, 1, 2, and 3 strong bipartite networks, respectively.}
    \label{fig:Case_1}
\end{figure}

\begin{itemize}
    \item Case 1: $0.5, 0.5, 0.5, 0.5, 0.5$ (5 weak signals)
    \item Case 2: $0.1, 0.5, 0.5, 0.5, 0.5$ (1 strong and 4 weak signals)
    \item Case 3: $0.1, 0.1, 0.5, 0.5, 0.5$ (2 strong and 3 weak signals)
    \item Case 4: $0.1, 0.1, 0.1, 0.5, 0.5$ (3 strong and 2 weak signals)
\end{itemize}

\begin{figure}[!t]
    \centering
    \includegraphics[width=0.85\textwidth]{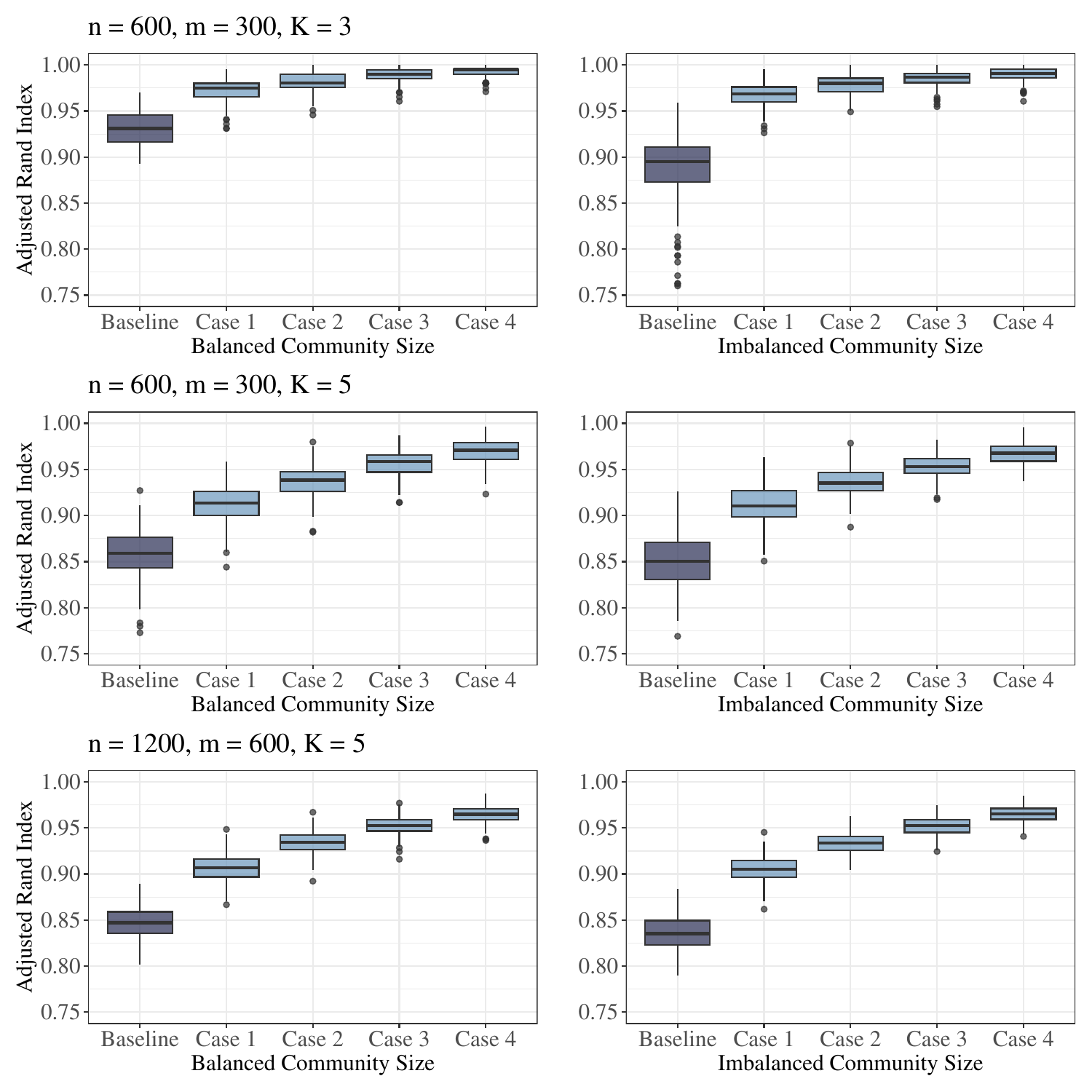}
    \caption{ARI for strong-signal primary networks. The baseline refers to the results obtained by using only the primary network. Cases 1, 2, 3, and 4 correspond to using 0, 1, 2, and 3 bipartite networks with the same signal strength as the primary network, respectively.}
    \label{fig:Case_2}
\end{figure}

The above four cases correspond to the use of 0, 1, 2, and 3 strong bipartite networks, respectively. The results of community detection are illustrated in Figure~\ref{fig:Case_1}, where the baseline SCORE method uses only the primary network. We observe obvious enhancement of ARI by incorporating bipartite information, for both balanced and imbalanced community sizes. The ARI of the baseline is around 0.25 in all four cases. This is consistent with our theoretical result (5) in Supplementary Material \citep{gao2026basic_supp} - for weak-signal primary networks, the mis-clustering rate easily diverges. Furthermore, the performance of BASIC is further improved as more strong-signal bipartite networks are incorporated. Last but not least, we observe that even if all five bipartite networks are weak-signal ones (Case 1), the enhancement is still obvious.

In the second scenario, we set the community structure signal of the primary network to be relatively strong, with the out-in ratio $\beta$ of the primary network being $0.3$. We aim to investigate whether incorporating the bipartite network can further enhance performance or, at the very least, prevent degradation. We set 0, 1, 2, and 3 bipartite networks with out-in ratios equal to that of the primary network, while the other bipartite networks set to $0.5$.

\begin{itemize}
    \item Case 1: $0.5, 0.5, 0.5, 0.5, 0.5$ (5 weak signals)
    \item Case 2: $0.3, 0.5, 0.5, 0.5, 0.5$ (1 strong and 4 weak signals)
    \item Case 3: $0.3, 0.3, 0.5, 0.5, 0.5$ (2 strong and 3 weak signals)
    \item Case 4: $0.3, 0.3, 0.3, 0.5, 0.5$ (3 strong and 2 weak signals)
\end{itemize}

The results are depicted in Figure~\ref{fig:Case_2}. The performances of the baseline in most cases are already fairly satisfactory due to the relatively strong signal from the primary network. However, leveraging information from bipartite networks can further enhance the performance of community detection in the primary network. Additionally, as shown in Case 1, even when the community structure of the bipartite networks is unclear, utilizing BASIC does not deteriorate the community detection. These phenomena are observed in both balanced and imbalanced cases and are consistent with the theoretical results in Theorem~A.2.

In the third scenario, we investigate the robustness of BASIC when the same-label assumption across the primary and bipartite networks is violated. In the previous settings, the bipartite networks are assumed to share the same community labels as the primary network. In this setting, we relax this assumption and allow the community labels in the bipartite networks to differ from those in the primary network.
Specifically, we consider one primary network and three bipartite networks, where the out-in ratios of the primary network and all bipartite networks are set to $\beta = 0.5$, corresponding to weak community signals. We fix the number of nodes in the primary network to $n = 600$, the number of nodes in each auxiliary node set to $m = 300$, and the number of communities to $K = 3$. Both balanced and imbalanced community size settings are considered.

To quantify label consistency, we define four cases according to the proportion of primary nodes in each bipartite network whose labels match those in the primary network:

\begin{itemize}
    \item Case 1: $1,1,1$ (100\% label consistency)
    \item Case 2: $0.9,0.9,0.9$ (90\% label consistency)
    \item Case 3: $0.7,0.7,0.7$ (70\% label consistency)
    \item Case 4: $0.5, 0.5, 0.5$ (50\% label consistency)
\end{itemize}
\noindent
These four cases represent increasing violations of the same-label assumption. We use this simulation to examine whether BASIC remains effective when the auxiliary bipartite networks only partially align with the primary network. The community detection results are shown in Figure~\ref{fig:Case_3-1}. We observe that BASIC remains effective when the same-label assumption is only mildly violated. In both balanced and imbalanced settings, the ARI in Case 1 is higher than that of the baseline, and the performance in Case 2 remains similar. Furthermore, as label consistency decreases, the performance of BASIC gradually deteriorates in both settings. Even in Cases 3 and 4, BASIC still outperforms the baseline, although the gain becomes much smaller when only 50\% of labels agree. Overall, these results suggest that BASIC is robust to moderate violations of the shared-label assumption, with misalignment leading to diminishing returns rather than clear performance degradation.

\begin{figure}
    \centering
    \includegraphics[width=0.85\linewidth]{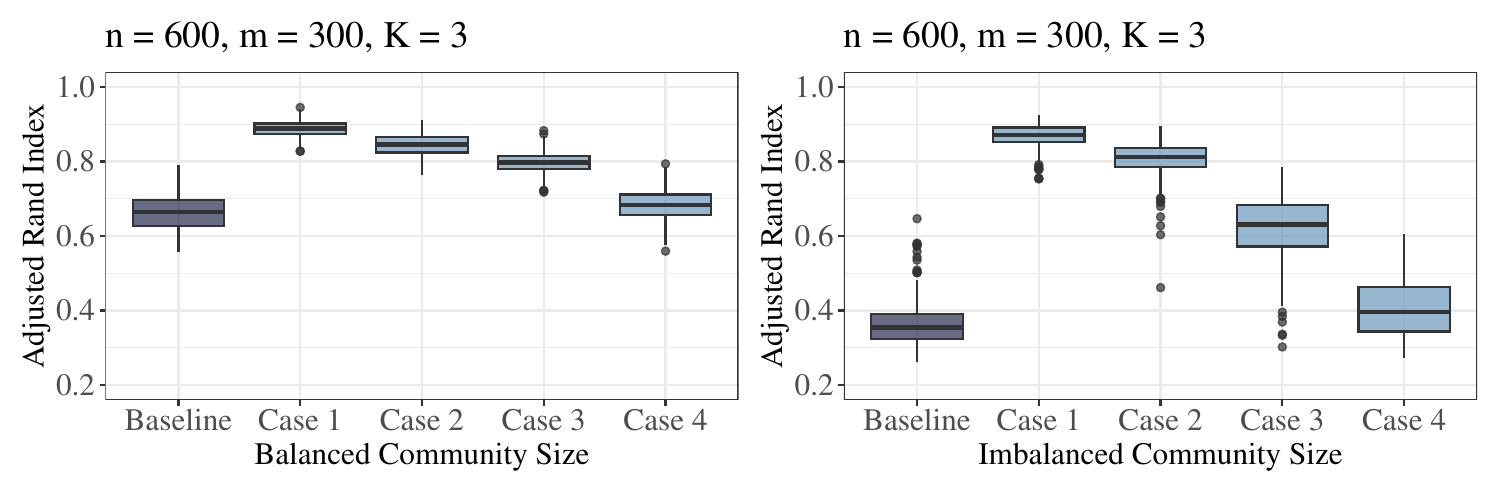}
    \caption{ARI under different levels of label consistency between the primary and bipartite networks. The baseline refers to the results obtained by using only the primary network. Cases 1--4 correspond to 100\%, 90\%, 70\%, and 50\% label consistency, respectively, between each bipartite network and the primary network.}
    \label{fig:Case_3-1}
\end{figure}

In the forth scenario, we compare BASIC with several existing methods. We consider one primary network and three bipartite networks, with the number of primary nodes $n = 600$, the number of auxiliary nodes $m = 300$, and the number of communities $K = 3$. Both balanced and imbalanced community size settings are included. We examine two signal regimes: a weak-signal setting with $\beta = 0.5$ for all networks, and a strong-signal setting with $\beta = 0.3$. We compare BASIC with SCORE and three representative competing methods, CAMSBM, MAS, and LSE, to assess the performance of BASIC relative to existing approaches under both weak- and strong-signal regimes. The results are shown in Figure~\ref{fig:Case_4-1}. It is observed that BASIC achieves the best performance among all methods in both weak- and strong-signal settings, for balanced and imbalanced community size scenarios.

\begin{figure}
    \centering
    \includegraphics[width=0.85\linewidth]{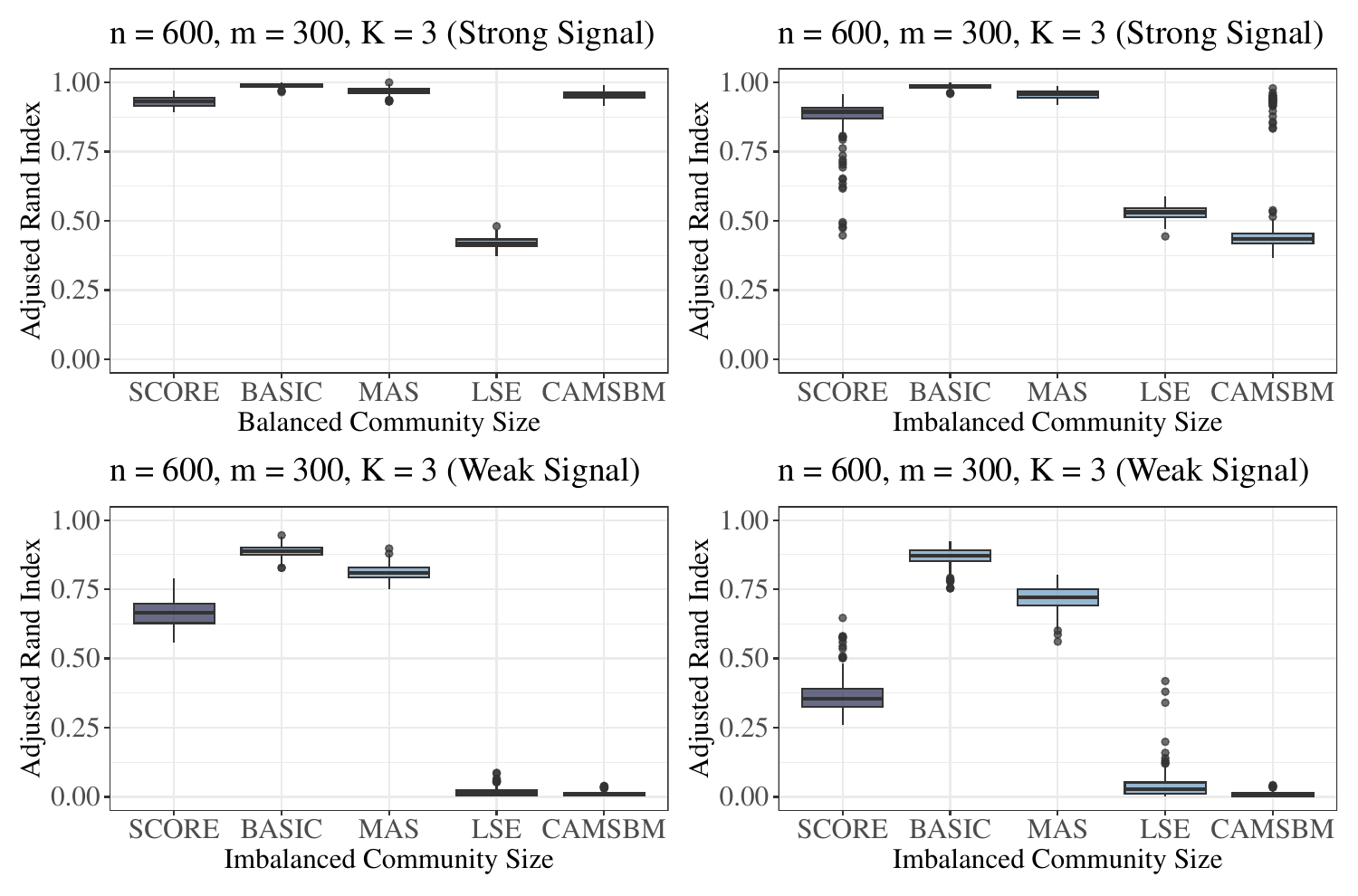}
    \caption{ARI comparison of BASIC and competing methods under weak- and strong-signal settings. The top panels correspond to strong-signal networks with $\beta=0.3$, and the bottom panels correspond to weak-signal networks with $\beta=0.5$. In each setting, both balanced and imbalanced community size scenarios are considered.}
    \label{fig:Case_4-1}
\end{figure}

Under the weak-signal setting, the advantage of BASIC is particularly obvious. Compared with SCORE, BASIC substantially improves the ARI in both balanced and imbalanced settings. BASIC also clearly outperforms MAS, while LSE and CAMSBM perform poorly in this setting, with ARI close to zero in most cases. This indicates that BASIC can effectively utilize bipartite information when the signal in each individual network is weak. Furthermore, under the strong-signal setting, the performance gap between BASIC and MAS becomes smaller, as both methods achieve high ARI values close to one. Nevertheless, BASIC still attains the best overall performance and remains more stable, especially under imbalanced community sizes.

In the last scenario, we consider a practical question. In real data, the author--paper bipartite network can be viewed as a weighted enhancement of the collaboration network, as it captures collaboration intensity. We therefore examine whether using only this weighted bipartite information as the auxiliary network can still improve performance. Specifically, we derive a binary primary network by thresholding the weighted author--author matrix induced from a bipartite network, and use the original bipartite network as the auxiliary bipartite network. We set $n = 600$, $m = 300$, and $K = 3$, and consider both balanced and imbalanced settings. The results in Figure~\ref{fig:Case_5-1} show that BASIC improves over the baseline method in both settings. In particular, the ARI increases substantially after incorporating the bipartite network, indicating that the weighted bipartite information provides additional signal beyond the thresholded primary network. These results suggest that, even when the bipartite network can be regarded as an enriched version of the primary network, incorporating it is not redundant but still enhances community detection performance.

\begin{figure}[!ht]
    \centering
    \includegraphics[width=0.55\linewidth]{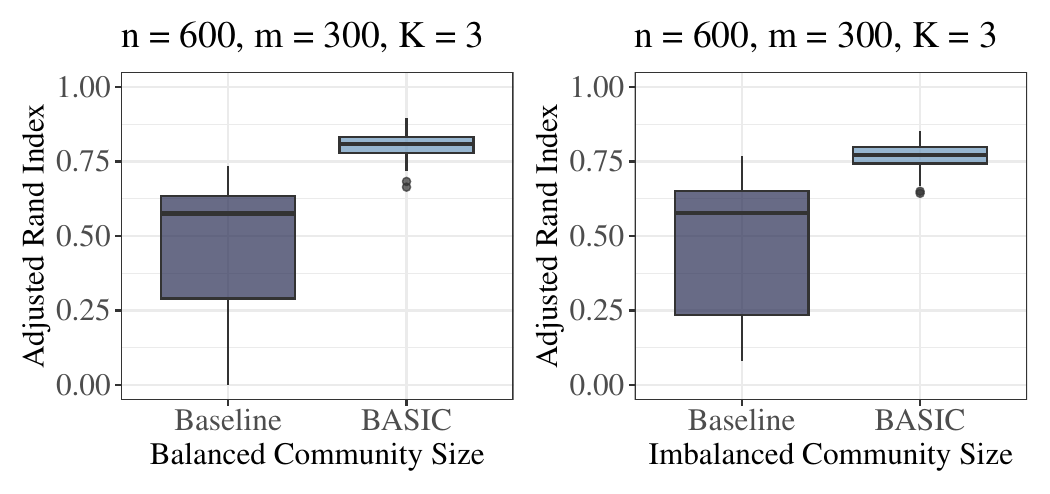}
    \caption{ARI for the weighted-enhancement scenario. The baseline refers to the results obtained by using only the thresholded binary primary network, while BASIC additionally incorporates the original bipartite network as auxiliary information.}
    \label{fig:Case_5-1}
\end{figure}

\section{Conclusion}\label{conclusion}

\subsection{Significance Statement}
In this paper, we propose a framework named  {\it Bipartite Assisted Spectral-clustering approach for Identifying Communities} (BASIC). It incorporates bipartite network information to community structure learning in social network analysis. By introducing an aggregated squared adjacency matrix, BASIC effectively takes into account the bipartite information without distorting the primary network structure. Empirical analysis on a self-collected large-scale academic collaboration dataset in the field of statistics further illustrates the effectiveness of the proposed method. By jointly modeling the collaboration, author–paper, author–institution, and author–region networks, BASIC reveals several research communities in the field of statistics, such as high-dimensional data analysis and biostatistics, which closely correspond to the main research directions and institutional clusters within the discipline. 
We further investigate the dynamic evolution of these communities over time. These results highlight how bipartite networks can provide complementary structural signals that improve the interpretability and stability of community detection in real-world networks. In addition, we also systematically study the theoretical properties of BASIC, and demonstrate its privilege across various scenarios of simulation studies. Overall, BASIC provides a practically effective and theoretically sound tool for integrating bipartite information to detect community structures in complex networks.

\subsection{Limitation and Future Work}
The present study has two main limitations. First, the temporal analysis is descriptive rather than based on a formal dynamic community model. Because BASIC is fitted separately to nested cumulative windows with growing node sets, the observed changes may reflect both community reorganization and the arrival of new authors and relationships. A dynamic DCBM/BiDCBM extension could explicitly model time-varying memberships, community birth, merger and disappearance, and structural change points. Second, the current theoretical framework is formulated for binary networks, and the weighted-input analysis serves only as an empirical sensitivity check. Developing a weighted version of BASIC with corresponding theoretical guarantees would allow collaboration intensity and time-decayed edge weights to be modeled directly and could ultimately be integrated with the dynamic framework.


\begin{acks}[Acknowledgments]
All authors contributed equally to this work and are listed in alphabetical order. All authors would like to thank the anonymous referees, an Associate Editor and the Editor for their constructive comments that improved the quality of this paper. 
\end{acks}

\begin{funding}
The research of Jingyuan Liu was supported by the National Key R\&D Program of China (No.~2022YFA1003800) and the National Natural Science Foundation of China (Nos.~12526213, 12271456, and 71988101). The research of Rui Pan was supported by the National Natural Science Foundation of China (No.~72471254). The research of Tianchen Gao was supported by the China Postdoctoral Science Foundation (No.~2026T190811). 
\end{funding}

\begin{supplement}
	\stitle{Appendix A: Theoretical Guarantees of BASIC and Appendix B: Proofs}
	\sdescription{
		Additional theoretical results establishing the
		theoretical guarantees of the proposed BASIC method,
		together with detailed proofs of the theoretical results
		presented in the main paper and Appendix A.
	}
\end{supplement}

\begin{supplement}
	\stitle{Data and code for BASIC}
	\sdescription{
		Source data and code for the BASIC algorithm, simulation experiments, and real-data analyses.
	}
\end{supplement}

\bibliographystyle{imsart-nameyear} 
\bibliography{ref}       

\end{document}